\documentclass[11pt]{amsart}
\usepackage{amsxtra, amsfonts, amsmath, amsthm, amstext, amssymb, amscd, mathrsfs, verbatim, color}
\usepackage{threeparttable}
\usepackage{mathtools}
\usepackage[ansinew]{inputenc}\usepackage[T1]{fontenc}
\usepackage[all,cmtip]{xy}
\theoremstyle{plain}

\newtheorem{theorem}{Theorem}[section]
\newtheorem{lemma}[theorem]{Lemma}
\newtheorem{definition-theorem}[theorem]{Definition-Theorem}
\newtheorem{proposition}[theorem]{Proposition}
\newtheorem{corollary}[theorem]{Corollary}

\newtheorem{conjecture}[theorem]{Conjecture}

\theoremstyle{definition}
\newtheorem{definition}[theorem]{Definition}
\newtheorem{example}[theorem]{Example}
\newtheorem{remark}[theorem]{Remark}
\newtheorem{notation}[theorem]{Notation}
\newcommand \bth[1] { \begin{theorem}\label{t#1} }
\newcommand \ble[1] { \begin{lemma}\label{l#1} }

\newcommand \bpr[1] { \begin{proposition}\label{p#1} }
\newcommand \bco[1] { \begin{corollary}\label{c#1} }
\newcommand \bde[1] { \begin{definition}\label{d#1}\rm }
\newcommand \bex[1] { \begin{example}\label{e#1}\rm }
\newcommand \bre[1] { \begin{remark}\label{r#1}\rm }

\newcommand \bnota[1] {\begin{notation}\label{n#1}\rm }
\newcommand {\ele} { \end{lemma} }

\newcommand {\epr} { \end{proposition} }
\newcommand {\eco} { \end{corollary} }
\newcommand {\ede} { \end{definition} }
\newcommand {\eex} { \end{example} }
\newcommand {\ere} { \end{remark} }
\newcommand {\enota} { \end{notation} }

\begin{document}

\title[Projectivity and indecomposability]{Homological branching law for $(\mathrm{GL}_{n+1}(F), \mathrm{GL}_n(F))$: projectivity and indecomposability} 

\author[Kei Yuen Chan]{Kei Yuen Chan}
\address{
Shanghai Center for Mathematical Sciences, Fudan University \\
}
\email{kychan@fudan.edu.cn}
\maketitle

\begin{abstract}
Let $F$ be a non-Archimedean local field. This paper studies homological properties of irreducible smooth representations restricted from $\mathrm{GL}_{n+1}(F)$ to $\mathrm{GL}_n(F)$. A main result shows that each Bernstein component of an irreducible smooth representation of $\mathrm{GL}_{n+1}(F)$ restricted to $\mathrm{GL}_n(F)$ is indecomposable. We also classify all irreducible representations which are projective when restricting from $\mathrm{GL}_{n+1}(F)$ to $\mathrm{GL}_n(F)$. A main tool of our study is a notion of left and right derivatives, extending some previous work joint with Gordan Savin. As a by-product, we also determine the branching law in the opposite direction.
\end{abstract}

\section{Introduction}

\subsection{}
Let $F$ be a non-Archimedean local field. Let $G_n=\mathrm{GL}_n(F)$. Let 
$\mathrm{Alg}(G_n)$ be the category of smooth representations of $G_n$, and all representations in this paper are in $\mathrm{Alg}(G_n)$. This paper is a sequel of \cite{CS18} in studying homological properties of smooth representations of $G_{n+1}$ restricted to $G_n$, which is originally motivated from the study of D. Prasad in his ICM proceeding \cite{Pr18}. In \cite{CS18}, we show that for generic representations $\pi$ and $\pi'$ of $G_{n+1}$ and $G_n$ respectively, the higher Ext-groups 
\[ \mathrm{Ext}^i_{G_n}(\pi, \pi')=0, \mbox{ for $i \geq 1$}, \]
 which was previously conjectured in \cite{Pr18}. This result gives a hope that there is an explicit homological branching law, generalizing the multiplicity one theorem \cite{AGRS}, \cite{SZ} and the local Gan-Gross-Prasad conjectures \cite{GGP09, GGP}.

The main techniques in \cite{CS18} are utilizing Hecke algebra structure developed in \cite{CS17, CS} and simultaneously applying left and right Bernstein-Zelevinsky derivatives, based on the classical approach of using Bernstein-Zelevinsky filtration on representations of $G_{n+1}$ restricted to $G_n$ \cite{Pr93, Pr18}. We shall extend these methods further, in combination of other things, to obtain new results in this paper.

The first part of the paper is to study projectivity under restriction. A supercuspidal representation of a reductive group $G$ is projective and, restricting to a closed subgroup $H$ of $G$, is still projective. It is natural to study the same property in the relative framework, that is to classify irreducible representations of $G$ restricted to $H$ is projective.

In \cite{CS18}, we showed that an essentially square-integrable representation $\pi$ of $G_{n+1}$ is projective when restricted to $G_n$. However, those representations do not account for all irreducible representations whose restriction is projective.  The first goal of the paper is to classify such representations:

\begin{theorem} \label{thm intro proj} (=Theorem \ref{thm class restrict proj})
Let $\pi$ be an irreducible smooth representation of $G_{n+1}$. Then $\pi|_{G_n}$ is projective if and only if 
\begin{enumerate}
\item $\pi$ is essentially square-integrable, or
\item $n+1$ is even, and $\pi \cong \rho_1 \times \rho_2$ for some cuspidal representations $\rho_1, \rho_2$ of $G_{(n+1)/2}$. 
\end{enumerate}
\end{theorem}
\noindent
There are also recent studies of the projectivity under restriction in other settings \cite{APS17}, \cite{La17} and \cite{CS18b}.

A main step in our classification is to the following projectivity criteria:

\begin{theorem}\label{thm proj critera} (=Theorem \ref{thm proj criteria} )
Let $\pi$ be an irreducible smooth representation of $G_{n+1}$. Then $\pi|_{G_n}$ is projective  if and only if the following two conditions hold:
\begin{enumerate}
\item  $\pi$ is generic; and 
\item $\mathrm{Hom}_{G_n}(\pi|_{G_n}, \omega) =0$ for any irreducible non-generic representation $\omega$ of $G_n$.
\end{enumerate}
\end{theorem}
Theorem \ref{thm proj critera} turns the projectivity problem into a Hom-branching problem. This is a consequence of two things: (1) the Euler-Poincar\'e pairing formula of D. Prasad \cite{Pr18} and (2) the Hecke algebra argument used in \cite{CS18} by G. Savin and the author. Roughly speaking, (1) is used to show non-projectivity while (2) is used to show projectivity. 

The second part of the paper studies  indecomposability of a restricted representation. It is clear that an irreducible representation (except one-dimensional ones) restricted from $G_{n+1}$ to $G_n$ cannot be indecomposable as it has more than one non-zero Bernstein components. However, the Hecke algebra realization in \cite{CS18, CS17} of the projective representations  in Theorem \ref{thm intro proj}  immediately implies that each Bernstein component of those restricted representation is indecomposable. This is a motivation of our study in general case, and precisely we prove:
\begin{theorem} \label{thm b component ind} (Theorem \ref{thm indecomp irred})
Let $\pi$ be an irreducible representation of $G_{n+1}$. Then for each Bernstein component $\tau$ of $\pi|_{G_n}$, any two non-zero $G_n$-submodules of $\tau$ have non-zero intersection.
\end{theorem}

As a consequence, we have:

\begin{corollary}
Let $\pi$ be an irreducible representation of $G_{n+1}$. Then any Bernstein component of $\pi|_{G_n}$ is indecomposable.
\end{corollary}

In Section \ref{ss remark cuspidal rep}, we explain how to determine which Bernstein component of $\pi$ is non-zero in terms of Zelevinsky segments, and hence Theorem \ref{thm b component ind} essentially parametrizes the indecomposable components of $\pi|_{G_n}$.

For a mirabolic subgroup $M_n$ of $G_{n+1}$, it is known \cite{Ze} that $\pi|_{M_n}$ is indecomposable for an irreducible representation $\pi$ of $G_{n+1}$. The approach in \cite{Ze} uses the Bernstein-Zelevinsky filtration of $\pi$ to $M_n$ and that 
 the bottom piece of the filtration is irreducible. We prove that the Bernstein component of a bottom piece is indecomposable as a $G_n$-module, and then make 
use of left and right derivatives, developed and used to prove main results in \cite{CS18}. The key fact is that left and right derivatives of an irreducible representation are asymmetric.  We now make more precise the meaning of 'asymmetric'. 
We say that an integer $i$ is the level of an irreducible representation $\pi$ of $G_n$ if the left derivative $\pi^{(i)}$ (and hence the right derivative ${}^{(i)}\pi$) is the highest derivative of $\pi$. 

\begin{theorem} \label{thm asymmetric derivative} (=Theorem \ref{lem same quo derivative})
Let $\pi$ be an irreducible smooth representation of $G_{n}$. Let $\nu(g)=|\mathrm{det}(g)|_F$. Suppose $i$ is not the level of $\pi$. Then $\nu^{1/2}\cdot\pi^{(i)}$ and $\nu^{-1/2}\cdot{}^{(i)}\pi$ have no isomorphic irreducible quotients whenever $\nu^{1/2}\cdot\pi^{(i)}$ and $\nu^{-1/2}\cdot{}^{(i)}\pi$ are non-zero. The statement also holds if one replaces quotients by submodules.

\end{theorem} 

\noindent

We remark that $\nu^{1/2}\cdot \pi^{(i)}$ and $\nu^{-1/2}\cdot {}^{(i)}\pi$ are shifted derivatives in the sense of \cite{Be84}, which has been used in loc. cit. to study unitary representations of $G_n$.

As a by-product of Theorem \ref{thm asymmetric derivative}, we give a complete answer to the branching law in another direction:

\begin{theorem} (=Corollary \ref{cor branching law direction})
Let $\pi_1, \pi_2$ be irreducible representations of $G_{n+1}$ and $G_n$ respectively. Then 
\[  \mathrm{Hom}_{G_n}(\pi_2, \pi_1|_{G_n}) \neq 0
\]
if and only if both $\pi_1$ and $\pi_2$ are one-dimensional and $\pi_2=\pi_1|_{G_n}$.

\end{theorem}

Computing the structure of a derivative of an arbitrary representation is a difficult question in general. Our approach is to approximate the information of derivatives of irreducible ones by some parabolically induced modules, whose derivatives can be computed via geometric lemma. On the other hand, the Speh representations behave more symmetrically for left and right derivatives, which motivates our proof to involve Speh representations.


Another key ingredient in proving Theorem \ref{thm b component ind} is a study on the submodule structure between left and right Bernstein-Zelevinsky inductions. We explain in Section \ref{sec submod structure} how the submodule structure of an induced module can be partly reflected from the module which induced from (see Proposition \ref{prop intersect same} for a precise statement). The study relies on the Hecke algebra structure of the Gelfand-Graev representation. 


\subsection{Organization of the paper}

Section \ref{bz der generic} studies derivatives of generic representations, which simplifies some computations for Theorem \ref{thm intro proj}. The results also give some guiding examples in the study of this paper and \cite{CS18}.

Section \ref{ss proj restrict} firstly proves a criteria for an irreducible representation to be projective under restriction, and then apply this criteria to give a classification of such class of modules.

Section \ref{sec submod structure} develops a theory of the submodule relation between left and right Bernstein-Zelevinsky inductions. 

Section \ref{sec indecomp}  proves a main result on the indecomposability of an irreducible representation under restriction, which uses results in Section \ref{sec submod structure} and the asymmetric property of derivatives proved in Section \ref{sec proof lemma}.

Section \ref{s preserve indecomp} proves that the Bernstein-Zelevinsky induction preserves indecomposability. This partly generalizes Section \ref{sec submod structure}.

In the first appendix, we explain how an irreducible representation appears as the unique submodule of a product of some Speh representations. In the second appendix, we provide some preliminaries on module theory.

\subsection{Acknowledgements} This article is a part of the project in studying Ext-branching laws, and the author would like to thank Gordan Savin for a number of helpful discussions. Part of the idea on indecomposability for restriction was developed during the participation in the program of 'On the Langlands Program: Endoscopy and Beyond' at IMS of National University of Singapore in January 2019. The author would like to thank the organizers for their warm hospitality. The author would like to thank the anonymous referee for useful comments.

\section{Bernstein-Zelevinsky derivatives of generic representations} \label{bz der generic}

\subsection{Notations}

Let $G_n=\mathrm{GL}_n(F)$. All representations in this paper are in $\mathrm{Alg}(G_n)$ and over $\mathbb C$, and we usually omit the adjective 'smooth'. Let $\mathrm{Irr}(G_n)$ be the set of all irreducible representations of $G_n$ and let $\mathrm{Irr}=\sqcup_{n} \mathrm{Irr}(G_n)$. For an admissible representation $\omega \in \mathrm{Alg}(G_n)$, denote by $\mathrm{JH}(\omega)$ the set of (isomorphism classes of) irreducible composition factors of $\omega$.

 Let $\rho$ be a cuspidal representation of $G_l$. Let $a, b\in \mathbb{C}$ with $b-a \in \mathbb Z_{\geq 0}$. We have a Zelevinsky segment $\Delta=[\nu^a\rho, \nu^b \rho]$, which we may simply call a segment. Denote $a(\Delta)=\nu^a\rho$ and $b(\Delta)=\nu^b\rho$. The relative length of $\Delta$ is defined as $b-a+1$ and the absolute length of $\Delta$ is defined as $l(b-a+1)$. We can truncate $\Delta$ form each side to obtain two segments of absolute length $r(b-a)$: 
\[ 
{}^{-}\Delta= [\nu^{a+1}\rho,  \ldots, \nu^{b}\rho ] \text{ and } \Delta^-= [ \nu^{a}\rho,  \ldots, \nu^{b-1}\rho ]. 
\] 
Moreover, if we perform the truncation $k$-times, the resulting segments will be denoted by ${}^{(kl)}\Delta$ and $\Delta^{(kl)}$. We remark that the convention here is different from the previous paper \cite{CS18} for convenience later. If $i$ is not an integer divisible by $l$, then we set $\Delta^{(i)}$ and ${}^{(i)}\Delta$ to be empty sets. We also denote $\Delta^{\vee}=[\nu^{-b}\rho^{\vee},\nu^{-a}\rho^{\vee}]$. For a singleton segment $[\rho, \rho]$, we abbreviate as $[\rho]$. For $\pi \in \mathrm{Alg}(G_l)$, define $n(\rho)=l$. For $c \in \mathbb C$ and a segment $\Delta=[\nu^a\rho, \nu^b \rho]$, define $\nu^c\Delta=[\nu^{a+c}\rho, \nu^{b+c}\rho]$.

For a Zelevinsky segment $\Delta$, define $\langle \Delta \rangle$ and $\mathrm{St}(\Delta)$ to be the (unique) irreducible submodule and quotient of $\nu^a\rho \times \ldots\times  \nu^b\rho$ respectively. We have 
\[ \mathrm{St}(\Delta)^{\vee} \cong \mathrm{St}(\Delta^{\vee}) \quad \mbox{and} \quad \langle \Delta \rangle^{\vee} \cong \langle \Delta^{\vee} \rangle .
\]

For two cuspidal representations $\rho_1, \rho_2$ of $G_m$, we say that $\rho_1$ precedes $\rho_2$, denoted by $\rho_1 < \rho_2$, if $\nu^c\rho_1 \cong \rho_2$ for some $c>0$. We say two segments $\Delta$ and $\Delta'$ are {\it linked} if $\Delta \not\subset \Delta$, $\Delta \not\subset \Delta'$ and $\Delta \cup \Delta'$ is still a segment. We say that a segment $\Delta$ {\it precedes} $\Delta'$, denoted by $\Delta < \Delta'$, if $b(\Delta)$ precedes $b(\Delta')$; and $\Delta$ and $\Delta'$ are linked. If $\Delta$ does not precede $\Delta'$, write $\Delta \not< \Delta'$. 

A multisegment is a multiset of segments. Let $\mathrm{Mult}$ be the set of multisegments. Let $\mathfrak{m}=\left\{ \Delta_1,\ldots, \Delta_r \right\} \in \mathrm{Mult}$. We relabel the segments in $\mathfrak{m}$ such that for $i<j$, $\Delta_i$ does not precede $\Delta_j$. The modules defined below are independent of the labeling (up to isomorphisms) \cite{Ze}. Define $\zeta(\mathfrak{m})=\langle \Delta_1 \rangle \times \ldots \times \langle \Delta_r \rangle$. Denote by $\langle \mathfrak{m} \rangle$ the unique irreducible submodule of $\zeta(\mathfrak{m})$. Similarly, define $\lambda(\mathfrak{m})=\mathrm{St}(\Delta_1)\times \ldots \times \mathrm{St}(\Delta_r)$. Denote by $\mathrm{St}(\mathfrak{m})$ the unique quotient of $\lambda(\mathfrak{m})$. Both notions $\langle \mathfrak m \rangle$ and $\mathrm{St}(\mathfrak m)$ give a classification of irreducible smooth representations of $G_n$ (\cite[Proposition 6.1]{Ze}, also see \cite{LM16}) i.e. both the maps
\[ \mathfrak m \mapsto \langle \mathfrak m \rangle, \quad \mbox{ and }\quad  \mathfrak m \mapsto \mathrm{St}(\mathfrak m) 
\]
determine bijections from $\mathrm{Mult}$ to $\mathrm{Irr}$. For example, when $\mathfrak m=\left\{[\nu^{-1/2}], [\nu^{1/2}] \right\}$, $\langle \mathfrak m \rangle$ is the Steinberg representation and $\mathrm{St}(\mathfrak m)$ is the trivial representation. The two notions $\langle \mathfrak m \rangle$ and $\mathrm{St}(\mathfrak m)$ are related by the so-called Aubert-Zelevinsky duality, and M\oe glin-Waldspurger algorithm. 

We shall use the following standard fact several times (see \cite[Theorems 1.9, 4.2 and 9.7]{Ze}): if two segments $\Delta, \Delta'$ are not linked, then 
\begin{align} \label{eqn commute product trivial} \langle \Delta \rangle \times \langle \Delta' \rangle \cong \langle \Delta' \rangle \times \langle \Delta \rangle, \end{align}
\begin{align} \label{eqn commute product}
 \mathrm{St}(\Delta) \times \mathrm{St}(\Delta') \cong \mathrm{St}(\Delta') \times \mathrm{St}(\Delta) .\end{align}

Let $\pi \in \mathrm{Irr}(G_n)$. Then $\pi$ is a subquotient of $\rho_1\times \ldots \times \rho_r$ for some irreducible cuspidal representations $\rho_i$ of $G_{n_i}$. The multiset $(\rho_1, \ldots, \rho_r)$, denoted by $\mathrm{csupp}(\pi)$, is called the cuspidal support of $\pi$. We also set
\[ \mathrm{csupp}_{\mathbb Z}(\pi)=\left\{ \nu^{c} \rho : \rho \in \mathrm{csupp}(\pi), c \in \mathbb{Z} \right\} , \]
which is regarded as a set (rather than a multiset).

\subsection{Derivatives and Bernstein-Zelevinsky inductions}
Let $U_n$ (resp. $U_n^-$) be the group of unipotent upper (resp. lower) triangular matrices in $G_n$. For $i\leq n$, let $P_i$ be the parabolic subgroup of $G_n$ containing the block diagonal matrices $\mathrm{diag}(g_1, g_2)$ ($g_1 \in G_i$, $g_2\in G_{n-i}$) and the upper triangular matrices. Let $P_i=M_iN_i$ with the Levi $M_i$ and the unipotent $N_i$. Let $N_i^-$ be the opposite unipotent subgroup of $N_i$. Let $\nu:G_n\rightarrow \mathbb{C}$ given by $\nu(g)=|\mathrm{det}(g)|_F$. Let 
\[  R_{n-i} =\left\{ \begin{pmatrix} g & x\\ 0 & u  \end{pmatrix} \in G_n : g \in \mathrm{GL}_{n-i}(F), u \in U_i, x \in \mathrm{Mat}_{n-i,i}(F) \right\} .
\]
Let $R_{n-i}^-$ be the transpose of $R_{n-i}$.

We shall use $\mathrm{Ind}$ for normalized induction and $\mathrm{ind}$ for normalized induction with compact support. Let $\psi_i$ be a character on $U_i$ given by $\psi_i(u)=\overline{\psi}(u_{1,2}+\ldots +u_{i-1,i})$, where $\overline{\psi}$ is a nondegenerate character on $F$ and $u_{k,k+1}$ is the value in the $(k,k+1)$-entry of $u$. For any irreducible representation $\tau$ of $G_{n-i}$, we extend trivially the $G_{n-i}\times U_i$-representation $\tau \boxtimes \psi_i$ to a $R_{n-i}$-representation. This defines functors from $\mathrm{Alg}(G_{n-i})$ to $\mathrm{Alg}(G_n)$ given by 
\[\pi \mapsto \mathrm{Ind}_{R_{n-i}}^{G_n}\pi \boxtimes \psi_i, \quad \mbox{ and } \pi \mapsto \mathrm{ind}_{R_{n-i}}^{G_n} \pi \boxtimes \psi_i ,\]
both of which will be called (right) Bernstein-Zelevinsky inductions. Similarly, one has left Bernstein-Zelevinsky inductions by using $R_{n-i}^-$ instead of $R_{n-i}$.

Let $\pi$ be a smooth representation of $G_n$. Following \cite{CS18}, define $\pi^{(i)}$ to be the left adjoint functor of $\mathrm{Ind}_{R_{n-i}}^{G_n} \pi \boxtimes \psi_i$. Let $\theta_n:G_n \rightarrow G_n$ given by $\theta_n(g)=g^{-T}$, the inverse transpose of $g$. Define the left derivative 
\begin{align} \label{eqn left derivative}
{}^{(i)}\pi:=\theta_{n-i}(\theta_n(\pi)^{(i)}),
\end{align}
which is left adjoint to $\mathrm{Ind}_{R_{n-i}^-}^{G_n}\pi \boxtimes \psi_i'$, where $\psi_i'(u)=\psi_i(u^T)$ for $u \in U_i^-$. The level of an admissible representation $\pi$ is the largest integer $i^*$ such that $\pi^{(i^*)}\neq 0$ and $\pi^{(j)}=0$ for all $j>i^*$. It follows from (\ref{eqn left derivative}) that if $i^*$ is the level of $\pi$, then ${}^{(i^*)}\pi \neq 0$ and ${}^{(j)}\pi=0$ for all $j>i^*$. When $i^*$ is the level for $\pi$, we shall call $\pi^{(i^*)}$ and ${}^{(i^*)}\pi$ to be the highest left and right derivatives of $\pi$ respectively, where we usually drop the term of left and right if no confusion.

We now define the shifted derivatives as follow: for any $i$,
\[  \pi^{[i]} = \nu^{1/2} \cdot \pi^{(i)} , \quad \mbox{ and }\quad  {}^{[i]}\pi = \nu^{-1/2} \cdot {}^{(i)}\pi .\]
For the details of Bernstein-Zelevinsky filtrations, see \cite{CS18}.

\subsection{On computing derivatives} \label{ss bz expressions}

Let $\Pi_i=\mathrm{ind}_{U_i}^{G_i}\psi_i$ be the Gelfand-Graev representation. Using inductions by stages, we have that
\[  \mathrm{ind}_{R_{n-i}}^{G_n} \pi \boxtimes \psi_i \cong \pi \times \mathrm{ind}_{U_i}^{G_i} \psi_i =\pi \times \Pi_i.
\]
Let $\dot{w}_0 \in G_n$ whose antidiagonal entries are $1$ and other entries are $0$. By the left translation of $\dot{w}_0$ on $\mathrm{ind}_{R_{n-i}^-}^{G_n} \pi \boxtimes \psi_i$, we have that 
\[  \mathrm{ind}_{R_{n-i}^-}^{G_n} \pi\boxtimes \psi_i \cong \left( \mathrm{ind}_{U_i}^{G_i} \psi_i \right)\times \pi =\Pi_i \times \pi .
\]
We similarly have that 
\[ \mathrm{Ind}_{R_{n-i}}^{G_n} \pi \boxtimes \psi_i \cong \pi \times \mathrm{Ind}_{U_i}^{G_i}\psi_i .
\]
Since the right derivative is left adjoint to $\mathrm{Ind}_{R_{n-i}}^{G_n}$, we consequently have:
\[  \pi^{(i)} \cong (\pi_{N_i})_{U_i, \psi_i} ,
\]
where $U_i$ is regarded as the subgroup $G_{n-i}\times G_i$ via the embedding $g \mapsto \mathrm{diag}(1,g)$. We shall use the later expression when computing derivatives in Section \ref{ss classification}. We also have an analogous expression for left derivatives.

We shall often use the following lemma:

\begin{lemma}  \label{lem generic bz quotient}
Let $\pi \in \mathrm{Alg}(G_{n+1})$ and let $\pi'$ be an admissible representation of $G_n$. Suppose there exists $i$ such that  the following conditions hold:
\[ \mathrm{Hom}_{G_{n+1-i}}( \pi^{[i]}, {}^{(i-1)}\pi') \neq 0  ;
\]
and 
\[  \mathrm{Ext}^k_{G_{n+1-j}}( \pi^{[j]}, {}^{(j-1)}\pi'))=0
\]
for all $j=1,\ldots, i-1$ and all $k$. Then $\mathrm{Hom}_{G_n}(\pi,\pi')\neq 0$. 
\end{lemma}
\begin{proof}
The Bernstein-Zelevinsky filtration of $\pi$ gives that there exists 
\[  0 \subset \pi_n \subset \ldots \subset \pi_0 =\pi
\]
such that $\pi_{i-1}/\pi_{i} \cong \mathrm{ind}_{R_{n-i+1}}^{G_n} \pi^{[i]} \boxtimes \psi_i$. Now, by the second adjointness property of derivatives \cite[Lemma 2.2]{CS18}, we have
\[ \mathrm{Ext}^k_{G_n}(\mathrm{ind}_{R_{n-j+1}}^{G_{n}}\pi^{[j]} \boxtimes \psi_{j-1} , \pi') \cong \mathrm{Ext}^k_{G_{n-j+1}}( \pi^{[j]}, {}^{(j-1)}\pi') =0
\]
for all $k$ and $j$. Now a long exact sequence argument gives that
\begin{align}  \mathrm{dim}~\mathrm{Hom}_{G_{n}}(\pi, \pi') & \geq  \mathrm{dim}~\mathrm{Hom}_{G_{n}}(\mathrm{ind}_{R_{n+1-i}}^{G_{n}}\pi^{[i]} \boxtimes \psi_{i-1} , \pi' ) \\
  & = \mathrm{dim}~\mathrm{Hom}_{G_{n+1-i}}( \pi^{[i]}, {}^{(i-1)}\pi') \neq 0 .
\end{align}
\end{proof}

\subsection{Subrepresentation of a standard representation}


\begin{lemma} \label{lem sub rep full}
Let $\mathfrak m \in \mathrm{Mult}$. Suppose all segments in $\mathfrak m$ are singletons  i.e. of relative length $1$. Then $\lambda(\mathfrak m)$ has unique irreducible submodule and quotient. Moreover, the unique irreducible submodule is generic.
\end{lemma}
\begin{proof}
By definitions, $\lambda(\mathfrak{m})=\zeta(\mathfrak{m})$ and hence has unique submodule and quotient. Since all segments in $\mathfrak{m}$ are singletons, the submodule is generic \cite{Ze}.
\end{proof}


It is known (see \cite{JS83}\footnote{The author would like to thank D. Prasad for mentioning this reference in a disucssion.}) that $\lambda(\mathfrak{m})$ always has a generic representation as the unique submodule. We shall prove a slightly stronger statement, using Zelevinsky theory:

\begin{proposition} \label{thm subrep}
Let $\mathfrak{m} \in \mathrm{Mult}$. Then $\lambda(\mathfrak{m})$ can be embedded to $\lambda(\mathfrak m')$ for some $\mathfrak m' \in \mathrm{Mult}$ whose segments are singletons. In particular, $\lambda(\mathfrak{m})$ has a unique irreducible submodule and moreover, the submodule is generic.
\end{proposition}

\begin{proof}

Let $\rho$ be a cuspidal representation in $\mathfrak m$ such that for any cuspidal representation $\rho'$ in $\mathfrak m$, $\rho \not < \rho'$. Let $\Delta$ be a segment in $\mathfrak m$ with the shortest relative length among all segments $\Delta'$ in $\mathfrak m$ with $b(\Delta') \cong \rho$.

By definition of $\lambda(\mathfrak m)$, we have that
\[   \lambda(\mathfrak m) \cong  \mathrm{St}(\Delta) \times  \lambda(\mathfrak m\setminus \left\{ \Delta \right\}) .
\]
On the other hand, we have that 
\[ \mathrm{St}(\Delta) \hookrightarrow b(\Delta) \times \mathrm{St}(\Delta^-) .
\]
Thus we have that
\begin{align} \label{eqn embedding sec 2}
 \lambda(\mathfrak m) & \hookrightarrow b(\Delta) \times \mathrm{St}(\Delta^-)\times  \lambda(\mathfrak m\setminus \left\{ \Delta \right\})  \\ \label{eqn iso segments}
& \cong b(\Delta) \times \lambda(\mathfrak m\setminus \left\{ \Delta \right\}+\Delta^-) 
\end{align}
We now explain the last isomorphism (\ref{eqn iso segments}), and it suffices to show 
\[  \lambda(\mathfrak m \setminus \left\{ \Delta \right\}+\Delta^-) \cong \mathrm{St}(\Delta^-) \times \lambda(\mathfrak m \setminus \left\{ \Delta \right\}) .
\]
To this end, we write $\mathfrak m=\left\{ \Delta, \Delta_1, \ldots, \Delta_k, \Delta_{k+1}, \ldots, \Delta_s \right\}$ such that 
\[ b(\Delta) \cong b(\Delta_1) \cong \ldots \cong b(\Delta_k)
\]
and $b(\Delta_j) \not\cong b(\Delta)$ for $j \geq k+1$. Then we have that 
\begin{align*}
\mathrm{St}(\Delta^-)\times \lambda(\mathfrak m \setminus \left\{\Delta \right\} ) \cong & \mathrm{St}(\Delta^-)\times \mathrm{St}(\Delta_1)\times \ldots \times \mathrm{St}(\Delta_s) \\
 \cong & \mathrm{St}(\Delta_1)\times \ldots \times \mathrm{St}(\Delta_k)\times \mathrm{St}(\Delta^-) \times \mathrm{St}(\Delta_{k+1})\times \ldots \times \mathrm{St}(\Delta_s) \\
\cong & \lambda(\mathfrak m\setminus \left\{\Delta \right\}+\Delta^-)
\end{align*}
where the second isomorphism follows from applying (\ref{eqn commute product}) few times. Here we need to use our choice of $\Delta$, which guarantees that $\Delta^-$ and $\Delta_j$ are unlinked for $j=1, \ldots, k$.


Now $\lambda(\mathfrak m\setminus \left\{ \Delta \right\} +\Delta^-)$ embeds to $\lambda(\mathfrak m')$ by induction for some $\mathfrak m'\in \mathrm{Mult}$ with all segments to be singletons. Thus this gives that $b(\Delta) \times \lambda(\mathfrak m\setminus \left\{ \Delta \right\}+\Delta^-)$ embeds to $b(\Delta)\times \lambda(\mathfrak m') \cong \lambda(\mathfrak m'+b(\Delta))$, and so does $\lambda(\mathfrak m)$ by (\ref{eqn embedding sec 2}).

The second assertion follows from Lemma \ref{lem sub rep full}.

\end{proof}

\subsection{Derivatives of generic representations}

Recall that the socle (resp. cosocle) of an admissible representation $\pi$ of $G_n$ is the maximal semisimple submodule (resp. quotient) of $\pi$.
\begin{lemma} \label{lem coscole socle}
Let $\pi \in \mathrm{Irr}(G_n)$. Then the cosocle of $\pi^{(i)}$ (resp. ${}^{(i)}\pi$) is isomorphic to the socle of $\pi^{(i)}$ (resp. ${}^{(i)}\pi$). 
\end{lemma}
\begin{proof}
This is almost the same as the proof of \cite[Lemma 2.2]{CS18}. More precisely, it follows  from that for an irreducible $G_n$-representation $\pi$,
\[ {}^{(i)}\pi \cong \theta_{n-i}(\theta_n(\pi)^{(i)}) \cong \theta_{n-i} ((\pi^{\vee})^{(i)})\cong \theta_{n-i}(({}^{(i)}\pi)^{\vee} )\]
and the fact \cite{BZ1} that $\theta_{n-i}(\tau)\cong \tau^{\vee}$ for any irreducible $G_{n-i}$-representation $\tau$. The last isomorphism follows from \cite[Lemma 2.2]{CS18}.
\end{proof}

\begin{proposition} \label{prop multi free}
Let $\pi \in \mathrm{Irr}(G_{n+1})$. The socle and cosocle of $\pi^{(i)}$ (and ${}^{(i)}\pi$) are multiplicity-free. 
\end{proposition}

\begin{proof}
 Let $\pi_0$ be an irreducible quotient of $\pi^{(i)}$. Let $\pi_1$ be a cuspidal representation of $G_{i-1}$ which is not in $\mathrm{cupp}_{\mathbb Z}(\pi_0)$. Then, by comparing cuspidal supports, 
\[ \mathrm{Ext}^j_{G_{n+1-k}}( \pi^{[k]}, {}^{(k-1)}(\pi_0 \times \pi_1))=0 \]
 for all $j$ and $k<i$. With a long exact sequence argument using Bernstein-Zelevinsky filtration (similar to the proof of Lemma \ref{lem generic bz quotient}), we have that
\begin{align*}
 \dim~ \mathrm{Hom}_{G_n}(\pi, \pi_0\times \pi_1) &\geq \mathrm{dim}~ \mathrm{Hom}_{G_{n+1-i}}(\pi^{[i]},{}^{(i-1)}(\pi_0\times \pi_1)) 
\end{align*}
Now one applies the geometric lemma to obtain a filtration on ${}^{(i-1)}(\pi_0\times \pi_1)$, and then by comparing cuspidal supports, the only possible layer that can contribute the above non-zero Hom is $\pi_0$. Hence,
\begin{align*}
 \dim~ \mathrm{Hom}_{G_n}(\pi, \pi_0\times \pi_1)& \geq \mathrm{dim}~\mathrm{Hom}_{G_{n+1-i}}( \pi^{[i]}, \pi_0). 
\end{align*}
 The first dimension is at most one by \cite{AGRS} and so is the second dimension. This implies the cosocle statement by Lemma \ref{lem generic bz quotient} and the socle statement follows from Lemma \ref{lem coscole socle}.
\end{proof}

A representation $\pi$ of $G_n$ is called {\it generic} if $\pi^{(n)}\neq 0$. The Zelevinsky classification of irreducible generic representations is in \cite{Ze}, that is, $\mathrm{St}(\mathfrak{m})$ is generic if and only if any two segments in $\mathfrak{m}$ are unlinked. With Proposition \ref{prop multi free}, the following result essentially gives a combinatorial description on the socle and cosocle of the derivatives of a generic representation.

\begin{corollary} \label{cor generic quo sub}
Let $\pi \in \mathrm{Irr}(G_{n+1})$ be generic. Then any simple quotient and submodule of $\pi^{(i)}$ (resp. ${}^{(i)}\pi$)  is generic. 
\end{corollary}

\begin{proof}
By Lemma \ref{lem coscole socle}, it suffices to prove the statement for quotient. Let $\mathfrak{m}=\left\{\Delta_1, \ldots, \Delta_r \right\}$ be the Zelevinsky segment $\mathfrak{m}$ such that 
\[ \pi \cong \mathrm{St}(\mathfrak{m})= \lambda(\mathfrak{m})=\mathrm{St}(\Delta_1)\times \ldots \times \mathrm{St}(\Delta_r) . \]
Since any two segments in $\mathfrak{m}$ are unlinked, we can label in any order and so we shall assume that for $i<j$, $b(\Delta_j)$ does not precede $b(\Delta_i)$. Then geometric lemma produces a filtration on $\pi^{(i)}$ whose successive subquotient is isomorphic to 
\[ \mathrm{St}({}^{(i_1)}\Delta_1)\times \ldots \times \mathrm{St}({}^{(i_r)}\Delta_r) , \]
where $i_1+\ldots+i_r=i$. The last module is isomorphic to $\lambda(\mathfrak{m}')^{\vee}$, where 
\[ \mathfrak{m}'=\left\{ ({}^{(i_1)}\Delta_1)^{\vee}, \ldots, ({}^{(i_r)}\Delta_r)^{\vee} \right\} . \]
If $\pi'$ is a simple quotient of $\pi^{(i)}$, then $\pi'$ is a simple quotient of one successive subquotient in the filtration, or in other words, $\pi'$ is a simple submodule of $\lambda(\mathfrak{m}')$ for a multisegment $\mathfrak{m}'$. Now the result follows from Proposition \ref{thm subrep}.
\end{proof}

\begin{remark}
One can formulate the corresponding statement of Proposition \ref{prop multi free} for affine Hecke algebra level by using the sign projective module in \cite{CS17}. Then it might be interesting to ask for an analogue result for affine Hecke algebra over fields of positive characteristics.
\end{remark}

Here we give a consequence to branching law (c.f. \cite{GGP}, \cite{Gu18}, \cite{Ch20}):

\begin{corollary}
Let $\pi \in \mathrm{Irr}(G_{n+1})$ be generic. Let $\pi' \in \mathrm{Irr}(G_n)$ and let $\mathfrak m \in \mathrm{Mult}$ with $\pi'\cong \langle \mathfrak m\rangle$. If $\mathrm{Hom}_{G_n}(\pi,\langle \mathfrak{m} \rangle ) \neq 0$, then each segment in $\mathfrak{m}$ has relative length at most $2$.
\end{corollary}

\begin{proof}
Write $\mathfrak{m}=\left\{ \Delta_1,\ldots, \Delta_r\right\} $ such that $\Delta_i$ does not precede $\Delta_j$ if $i <j$. Let $\pi_0=\langle \mathfrak m \rangle$. By using the Bernstein-Zelevinsky filtration, $\mathrm{Hom}_{G_n}(\pi, \pi_0 ) \neq 0$ implies that  
\[ \mathrm{Hom}_{G_{n+1-i}}( \pi_1^{[i]}, {}^{(i-1)}\pi_0) \neq 0 \]
 for some $i \geq 1$ \cite{CS18} (c.f. Lemma \ref{lem generic bz quotient}). Hence Corollary \ref{cor generic quo sub} implies that ${}^{(i-1)}\pi_0$ is generic for some $i \geq 1$. Since
\[  \pi_0 \hookrightarrow \zeta(\mathfrak m) \]
we have ${}^{(i-1)}\pi_0 \hookrightarrow {}^{(i-1)}\zeta(\mathfrak m)$ and so ${}^{(i-1)}\zeta(\mathfrak m)$ has a generic composition factor. On the other hand, geometric lemma gives that  ${}^{(i-1)}\pi_0$ admits a filtration whose successive quotients are isomorphic to $ {}^{(i_1)}\langle \Delta_1  \rangle \times \ldots \times {}^{(i_r)} \langle \Delta_r \rangle$ where $i_1, \ldots, i_r$ run  for all sums equal to $i-1$. Then at least one such quotient is non-degenerate and so in that quotient, all ${}^{(i_k)}\langle \Delta_k \rangle$ are cuspidal. Following from the derivatives on $\langle \Delta \rangle$, $\Delta$ can have at most of relative length $2$. 
\end{proof}

Another consequence is on the indecomposability of derivatives of generic representations.

\begin{corollary}
Let $\pi \in \mathrm{Irr}(G_{n+1})$ be generic. Then the projections of $\pi^{(i)}$ and ${}^{(i)}\pi$ to any cuspidal support component have unique simple quotient and submodule. In particular, the projections of $\pi^{(i)}$ and ${}^{(i)} \pi$ to any cuspidal support component are indecomposable.
\end{corollary}

\begin{proof}
For a fixed cuspidal support, there is a unique (up to isomorphism) irreducible smooth generic representation. Now the result follows from Proposition \ref{prop multi free} and Corollary \ref{cor generic quo sub}.
\end{proof}

The analogous statement is not true in general if one replaces an irreducible generic representation by an arbitrary irreducible representation. We give an example below.

\begin{example}
Let $\mathfrak m=\left\{ [1,\nu],[\nu^{-1},1], [1], [\nu^{-1},\nu] \right\}$. We note that
\[ \langle \mathfrak m \rangle =\langle [1,\nu], [\nu^{-1},1]\rangle \times \langle [1] \rangle \times \langle [\nu^{-1},\nu]\rangle .
\]
We compute the components of $\langle \mathfrak m \rangle^{(2)}$, at the cuspidal support $\left\{ \nu^{-1},\nu^{-1}, 1,1, 1, \nu \right\}$. 
In this case, there are four composition factors whose Zelevinsky multisegments are:
\[   \mathfrak m_1= \left\{ [1], [\nu^{-1}] , [1], [\nu^{-1}, \nu]  \right\}, \quad \mathfrak m_2 =\left\{ [1,\nu], [\nu^{-1}], [1], [\nu^{-1},1] \right\}
\]
\[   \mathfrak m_3= \left\{ [1,\nu], [\nu^{-1},1], [\nu^{-1},1] \right\}  \mbox{ (multiplicity 2) }
\]
Note that the socle and cosocle coincide by Lemma \ref{lem coscole socle}. Note $\langle \mathfrak m_1\rangle$ and $\langle \mathfrak m_3 \rangle$ are in socle (and so cosocle). $\langle \mathfrak m_2\rangle$ cannot be a submodule or quotient. Thus the only possible structure is two indecomposable modules. One of them has composition factors of $\langle \mathfrak m_3 \rangle$ (with multiplicity 2) and $\langle \mathfrak m_2 \rangle$. Another one is a simple module isomorphic to $\langle \mathfrak m_1 \rangle$.

\end{example}

\section{Projectivity }\label{ss proj restrict}

\subsection{Projectivity Criteria}

We need the following formula of D. Prasad:

\begin{theorem}\cite{Pr18} \label{thm nongeneric}
Let $\pi_1$ and $\pi_2$ be admissible representations of $\mathrm{GL}_{n+1}(F)$ and $\mathrm{GL}_n(F)$ respectively. Then
\[ \sum_{i \in \mathbb Z} (-1)^i\mathrm{dim}~ \mathrm{Ext}^i_{G_n}(\pi_1, \pi_2) = \mathrm{dim}~ \mathrm{Wh}(\pi_1) \cdot \mathrm{dim}~\mathrm{Wh}(\pi_2) ,
\]
where $\mathrm{Wh}(\pi_1)=\pi_1^{(n+1)}$ and $\mathrm{Wh}(\pi_2)= \pi_2^{(n)}$.
\end{theorem}


\begin{lemma} \label{lem deriv gln}
Let $\pi$ be an irreducible $G_{n+1}$-representation. If $\pi^{(i)}$ has a non-generic irreducible submodule or quotient for some $i$, then there exists a non-generic representation $\pi'$ of $G_n$ such that $\mathrm{Hom}_{G_n}(\pi, \pi')\neq 0$. The statement still holds if we replace $\pi^{(i)}$ by ${}^{(i)}\pi$.
\end{lemma}

\begin{proof}
By Lemma \ref{lem coscole socle}, it suffices just to consider that $\pi^{(i)}$ has a non-generic irreducible quotient, say $\lambda$. Now let 
\[\pi'=(\nu^{1/2}\lambda) \times \tau , \]where $\tau$ is a cuspidal representation such that $\tau$ is not an unramified twist of a cuspidal representation appearing in any segment in $\mathfrak{m}$. Here $\mathfrak{m}$ is a multisegment with $\pi \cong \langle \mathfrak{m} \rangle$. Now 
\[\mathrm{Ext}^k_{G_{n+1-j}}(\pi^{[j]}, {}^{(j-1)}\pi')=0 \]
 for $j< n(\tau)$ and any $k$ since $\tau$ is in $\mathrm{cupp}({}^{(j-1)}\pi')$ whenever it is nonzero while $\tau \not\in \mathrm{cupp}( \pi^{[j]})$. Moreover, ${}^{(n(\tau)-1)}\pi' $ has a simple quotient isomorphic to $\nu^{1/2}\lambda$. This checks the Hom and Ext conditions in Lemma \ref{lem generic bz quotient} and hence proves the lemma. The proof for ${}^{(i)}\pi$ is almost identical with switching left and right derivatives in suitable places.
\end{proof}

\begin{theorem} \label{thm proj criteria} 
Let $\pi\in \mathrm{Irr}(G_{n+1})$. Then the following conditions are equivalent:
\begin{enumerate}
\item $\pi|_{G_n}$ is projective 
\item $\pi$ is generic and any irreducible quotient of $\pi|_{G_n}$ is generic.
\end{enumerate}
\end{theorem}

\begin{proof}
For (2) implying (1), it is proved in \cite{CS18}. We now consider $\pi$ is projective. All higher Exts vanish and so $\mathrm{EP}(\pi, \pi')=\mathrm{dim}~\mathrm{Hom}_{G_n}(\pi,\pi')$ for any irreducible $\pi'$ of $G_n$. If $\pi'$ is an irreducible quotient of $\pi$, then $\mathrm{EP}(\pi,\pi')\neq 0$ and hence $\pi$ is generic by Theorem \ref{thm nongeneric}. But Theorem \ref{thm nongeneric} also implies $\pi'$ is generic. This proves (1) implying (2).

\end{proof}

\subsection{Classification} \label{ss classification}

\begin{definition} \label{def restricted proj}
 We say that an irreducible representation $\pi$ of $G_{n+1}$ is {\it relatively projective} if either one of the following conditions holds:
\begin{enumerate}
\item[(i)] $\pi$ is essentially square-integrable;
\item[(ii)] $\pi$ is isomorphic to $\pi_1 \times \pi_2$ for some cuspidal representations $\pi_1, \pi_2$ of $G_{(n+1)/2}$ with $\pi_1 \not\cong \nu^{\pm 1}\pi_2$. 
\end{enumerate}
In particular, a relatively-projective representation is generic. The condition $\pi_1 \not\cong \nu^{\pm 1}\pi_2$ is in fact automatic from $\pi$ being irreducible.

We can formulate the conditions (i) and (ii) combinatorially as follows. Let $\pi \cong \mathrm{St}(\mathfrak{m})$ for a multisegment $\mathfrak{m}=\left\{\Delta_1, \ldots, \Delta_r\right\}$. Then (i) is equivalent to  $r=1$; and (ii) is equivalent to that $r=2$, and $\Delta_1$ and $\Delta_2$ are not linked, and the relative lengths of $\Delta_1$ and $\Delta_2$ are both $1$. 
\end{definition}

\begin{lemma} \label{lem generic bz quotient gp}
Let $\pi \in \mathrm{Irr}(G_{n+1})$ be not relatively projective. Then there exists an irreducible non-generic representation $\pi'$ of $G_{n}$ such that $\mathrm{Hom}_{G_n}(\pi, \pi') \neq 0$. 
\end{lemma}

\begin{proof}

It suffices to construct an irreducible non-generic representation $\pi'$ satisfying the Hom and Ext properties in Lemma \ref{lem generic bz quotient}.

Let $\mathfrak{m}=\left\{ \Delta_1,\ldots, \Delta_r \right\}$ be a multisegment such that $\pi \cong \mathrm{St}(\mathfrak{m})$. We divide into few cases.

{\bf Case 1:} $r\geq 3$; or when $r=2$, each segment has relative length at least $2$; or when $r=2$, $\Delta_1 \cap \Delta_2=\emptyset$. We choose a segment $\Delta'$ in $\mathfrak{m}$ with the shortest absolute length. Now we choose a maximal segment $\Delta$ in $\mathfrak{m}$ with the property that $\Delta' \subset \Delta$. By genericity, $\nu^{-1}a(\Delta) \notin \Delta_k$ for any $\Delta_k\in \mathfrak{m}$. Let 
\[\mathfrak{m}'=\left\{ \nu^{1/2}\Delta, [\nu^{-1/2}a(\Delta)], [\tau] \right\} , \]
 where $\tau$ is a cuspidal representation so that $\mathrm{St}(\mathfrak{m}')$ is a representation of $G_n$ and $\tau$ is not an unramified twist of any cuspidal representation appearing in a segment of $\mathfrak{m}$. To make sense of the construction, it needs the choices and the assumptions on this case. Let $k=n(a(\Delta)), l=n(\tau)$. Let 
\[ \pi'=\mathrm{St}(\mathfrak{m}') . \]
Now as for $i <k+l+1$, either $\nu^{-1/2}a(\Delta)$ or $\tau$ appears in the cuspidal support of ${}^{(i-1)}\pi$, but not in that of $\nu^{1/2}\pi^{(i)}$, 
\[ \mathrm{Ext}^j_{G_{n+1-i}}(\nu^{1/2}\cdot \pi^{(i)}, {}^{(i-1)}\pi') = 0 \]
and all $j$. By Corollary \ref{cor generic quo sub}, $\nu^{1/2}\cdot \pi^{(k+l+1)}$ has a simple generic quotient isomorphic to $\nu^{1/2}\mathrm{St}({}^-\Delta)$. On the other hand, a submodule structure of ${}^{(k+l)}\pi'$ can be computed as follows: 
\begin{align}  \label{eqn derivative compute}
 0 \neq     &\mathrm{Hom}_{G_n}(\lambda(\mathfrak m'), \mathrm{St}(\mathfrak m')) \\
\cong &  \mathrm{Hom}_{G_{n-k-l} \times G_{k+l}}(  \mathrm{St}(\nu^{1/2}\cdot{}^-\Delta)  \boxtimes (\nu^{-1/2}a(\Delta) \times \tau), \mathrm{St}(\mathfrak m')_{N_{k+l}^-}) , 
\end{align}
 Here the non-zeroness comes from the fact that $\mathrm{St}(\mathfrak m')$ is the unique quotient of $\lambda(\mathfrak m')$, and the isomorphism follows from Frobenius reciprocity. Since taking the derivative is an exact functor, we have that $ \mathrm{St}(\nu^{1/2}\cdot {}^-\Delta) $ is a subrepresentation of ${}^{(k+l)}\mathrm{St}(\mathfrak m')$ (see Section \ref{ss bz expressions}). Thus we have 
\[ \mathrm{Hom}_{G_{n-k-l}}( \pi^{[k+l+1]}, {}^{(k+l)}\pi') \neq 0 .\]

{\bf Case 2}: $r=2$ with $\Delta_1 \cap \Delta_2 \neq \emptyset$ and one segment having relative length $1$ (and not both having relative length $1$ by the definition of relatively-projective type). By switching the labeling on segments if necessary, we assume that $\Delta_1 \subset \Delta_2$. Let $p$ and let $l$ be the absolute and relative length of $\Delta_2$ respectively. Let 
\[ \mathfrak{m}'=\left\{  [\nu^{1/2}a(\Delta_1)], [\nu^{3/2-l}a(\Delta_1), \nu^{-1/2}a(\Delta_1)], [\tau] \right\} ,  \hspace{0.2 cm} \pi'=\mathrm{St}(\mathfrak{m}') ,\] where $\tau$ is a cuspidal representation of $G_{k}$ (here $k$ is possibly zero) so that $\mathrm{St}(\mathfrak{m}')$ is a $G_n$-representation. Note that $\pi'$ is non-generic. By Corollary \ref{cor generic quo sub} and geometric lemma, a simple quotient $\nu^{1/2}\cdot \pi^{(p)}$ is isomorphic to $\nu^{1/2}a(\Delta_1)$. Similar computation as in (\ref{eqn derivative compute}) gives that a simple module of ${}^{(p-1)}\pi'$ is isomorphic to $\nu^{1/2}a(\Delta_1)$. This implies the non-vanishing $\mathrm{Hom}$ between those two $G_{n+1-p}$-representations.

 We now prove the vanishing $\mathrm{Ext}$-groups in order to apply Lemma \ref{lem generic bz quotient}. Now applying the Bernstein-Zelevinsky derivatives $(j=1,\ldots, p-1)$, unless $a(\Delta_1)\cong b(\Delta_2)$, we have that $\nu^{3/2-l}a(\Delta_1)$ is a cuspidal support for ${}^{(j-1)}\pi'$ whenever ${}^{(j-1)}\pi'$ is nonzero and is not a cuspidal support for $\nu^{1/2}\pi^{(j)}$. It remains to consider $a(\Delta_1)\cong b(\Delta_2)$. We can similarly consider the cuspidal support for $\nu^{-1/2}a(\Delta_1)$ and $\nu^{1/2}a(\Delta_1)$ to make conclusion.
\end{proof}

\begin{lemma} \label{lem if direction}
Let $\pi \in \mathrm{Irr}(G_{n+1})$. If $\pi$ is (generic) relatively-projective, then $\pi|_{G_n}$ is projective.
\end{lemma}

\begin{proof}
When $\pi$ is essentially square-integrable, it is proved in \cite{CS18}. We now assume that $\pi$ is in the case (2) of Definition \ref{def restricted proj}. It is equivalent to prove the condition (2) in Theorem \ref{thm proj criteria}. Let $\pi' \in \mathrm{Irr}(G_n)$ with $\mathrm{Hom}_{G_n}(\pi,\pi')\neq 0$. We have to show that $\pi'$ is generic. Note that the only non-zero derivative of $\pi^{(i)}$ can occur when $i=n+1$ and $\frac{n+1}{2}$.

{\bf Case 1:} $\mathrm{Hom}_{G_{(n+1)/2}}(\pi^{[(n+1)/2]}, {}^{((n-1)/2)}\pi') \neq 0$. For (1), let 
\[ \pi = \rho_1 \times \rho_2\]
for some irreducible cuspidal representations $\rho_1, \rho_2$ of $G_{(n+1)/2}$ with $\rho_1 \not \cong \nu^{\pm 1}\rho_2$, and 
\[ \mathfrak{m}'=\left\{ \Delta_1',\ldots, \Delta_s' \right\} \quad \mbox{ for $\pi'\cong \mathrm{St}(\mathfrak{m}')$} . \]
 By a simple count on dimensions, we must have $\Delta_k' \cong \nu^{1/2}\rho_1$ or $ \cong \nu^{1/2}\rho_2$ for some $k$. Using dimensions again, we have for $l\neq k$, $[\rho_1]$ and $[\rho_2]$ are unlinked to $\Delta_l'$ and so 
\[\pi' \cong (\nu^{1/2}\cdot \rho_r)\times \mathrm{St}(\mathfrak{m}'\setminus \left\{ \Delta_k \right\}) ,\]
for $r=1$ or $2$. Then 
\[ {}^{((n-1)/2)}\pi'\cong\nu^{1/2}\cdot \rho_r , \]
which implies 
\[ {}^{((n-1)/2)}\mathrm{St}(\mathfrak{m}'\setminus \left\{ \Delta_k' \right\})\neq 0 .\]
 Thus $\mathrm{St}(\mathfrak{m}'\setminus \left\{ \Delta_k' \right\})$ is generic and so is $\pi'$. 

{\bf Case 2:}  $\mathrm{Hom}_{G_{(n+1)/2}}( \pi^{[(n+1)/2]}, {}^{(n-1)/2}\pi') = 0$. We must have 
\[ \mathrm{Hom}_{G_{n+1-i}}(\pi^{[n+1]}, {}^{(n)}\pi')\neq 0 \]
 and so ${}^{(n)}\pi'\neq 0 $. Hence $\pi'$ is generic. 
\end{proof}

We now achieve the classification of irreducible representations which are projective when restricted from $G_{n+1}$ to $G_n$:

\begin{theorem} \label{thm class restrict proj}
Let $\pi \in \mathrm{Irr}(G_{n+1})$. Then $\pi|_{G_n}$ is projective if and only if $\pi$ is relatively projective in Definition \ref{def restricted proj}. 
\end{theorem}
 
\begin{proof}

The if direction is proved in Lemma \ref{lem if direction}. The only if direction follows from Lemma \ref{lem generic bz quotient gp} and Theorem \ref{thm nongeneric}.
\end{proof}

One advantage for such classification is that those restricted representations admit a more explicit realization as shown in \cite{CS18}:

\begin{theorem}
Let $\pi, \pi' \in \mathrm{Irr}(G_{n+1})$. If $\pi$ and $\pi'$ are relatively projective, then $\pi|_{G_n} \cong \pi'|_{G_n}$. In particular, $\pi|_{G_n}$ is isomorphic to the Gelfand-Graev representation $\mathrm{ind}_{U_n}^{G_n} \psi_n$. 
\end{theorem}

\begin{proof}
We have shown that $\pi$ and $\pi'$ have to be generic. Then we apply \cite[Corollary 5.5 and Theorem 5.6]{CS18}.
\end{proof}

\section{Submodule structure of Bernstein-Zelevinsky filtrations} \label{sec submod structure}

 Several insights come from the affine Hecke algebra realization of Gelfand-Graev representations. We shall first recall those results.

\subsection{Affine Hecke algebras} \label{ss affine hecke}

\begin{definition}
The affine Hecke algebra $\mathcal H_{l}(q)$ of type $A$ is an associative algebra over $\mathbb{C}$ generated by $\theta_1,\ldots, \theta_{l}$ and $T_w$ ($w \in S_{l}$) satisfying the relations:
\begin{enumerate}
\item $\theta_i\theta_j=\theta_j\theta_i$;
\item $T_{s_k}\theta_k-\theta_{k+1}T_{s_k}=(q-1)\theta_k$, where $q$ is a certain prime power and $s_k$ is the transposition between the numbers $k$ and $k+1$;
\item $T_{s_k}\theta_i=\theta_i T_{s_k}$, where $i \neq k, k+1$
\item $(T_{s_k}-q)(T_{s_k}+1)=0$;
\item $T_{s_k}T_{s_{k+1}}T_{s_k}=T_{s_{k+1}}T_{s_k}T_{s_{k+1}}$.
\end{enumerate}
Let $\mathcal A_l(q)$ be the (commutative) subalgebra generated by $\theta_1, \ldots, \theta_l$. Let $\mathcal H_{W,l}(q)$ be the subalgebra generated by $T_{s_1},\ldots, T_{s_{l-1}}$. Let $\mathrm{sgn}$ be the $1$-dimensional $\mathcal H_{W,l}(q)$-module characterized by $T_{s_k}$ acting by $-1$. 
\end{definition}

It is known from \cite[Proposition 3.11]{Lu89} that the center $\mathcal Z_l$ of $\mathcal H_l(q)$ has a basis $\left\{ z_M =\sum_{w \in S_l} \theta_1^{i_{w(1)}}\ldots \theta_l^{i_{w(l)}} \right\}_M$, where $M=(i_1, \ldots, i_l)$ runs for all $l$-tuples in $\mathbb{Z}^l/S_l$.

Bernstein decomposition asserts that
\[  \mathfrak{R}(G_n) \cong \prod_{\mathfrak{s} \in \mathfrak{B}(G_n)} \mathfrak{R}_{\mathfrak s}(G_n) ,\]
where $\mathfrak{R}(G_n)$ is the category of smooth $G_n$-representations, $\mathfrak{B}(G_n)$ is the set of inertial equivalence classes of $G_n$ and $\mathfrak R_{\mathfrak s}(G_n)$ is the full subcategory of $\mathfrak{R}(G_n)$ associated to $\mathfrak s$ (see \cite{BK2}). For a smooth representation $\pi$ of $G_n$, define $\pi_{\mathfrak s}$ to be the projection of $\pi$ to the component $\mathfrak R_{\mathfrak s}(G_n)$. 

For each $\mathfrak s \in \mathfrak B(G_n)$, \cite{BK} and \cite{BK2} associate with a compact group $K_{\mathfrak s}$ and a finite-dimensional representation $\tau$ of $K_{\mathfrak s}$, such that the convolution algebra \[ \mathcal H(K_{\mathfrak s}, \tau) := \left\{ f : G_n \rightarrow \mathrm{End}(\tau^{\vee}) : f(k_1gk_2) = \tau^{\vee}(k_1)\circ f(g) \circ \tau^{\vee}(k_2)  \mbox{ for $k_1,k_2 \in K_{\mathfrak s}$} \right\} \]
is isomorphic to the product $\mathcal H_{n_1}(q_1) \otimes \ldots \otimes \mathcal H_{n_r}(q_r)$ of affine Hecke algebra of type $A$, denoted by $\mathcal H_{\mathfrak s}$. For a smooth representation $\pi$ of $G_n$, the algebra $\mathcal H(K_{\mathfrak s},\tau)$ acts naturally on the space $\mathrm{Hom}_{K_{\mathfrak s}}(\tau, \pi) \cong (\tau^{\vee}\otimes \pi)^{K_{\mathfrak s}}$. This defines an equivalence of categories:
\begin{equation} \label{eqn equi cat h modules}
  \mathfrak{R}_{\mathfrak s}(G_n) \cong  \mbox{ category of $\mathcal H_{\mathfrak s}$-modules } .
\end{equation}
By abuse notation, we shall identify $\mathrm{Hom}_{K_{\mathfrak s}}(\tau, \pi)$ with $\pi_{\mathfrak s}$ under (\ref{eqn equi cat h modules}) and consider $\pi_{\mathfrak s}$ as $\mathcal H_{\mathfrak s}$-module.

Let $\mathcal A_{\mathfrak s} =\mathcal A_{n_1}(q_1)\otimes \ldots \otimes \mathcal A_{n_r}(q_r)$. Let $\mathcal H_{W,\mathfrak s}=\mathcal H_{W,n_1}(q_1)\otimes \ldots \otimes \mathcal H_{W, n_r}(q_r)$. Let 
\[\mathrm{sgn}_{\mathfrak s} =\mathrm{sgn}\boxtimes \ldots \boxtimes \mathrm{sgn}\]  as an $\mathcal H_{W,\mathfrak s}$-module. Note that in \cite{CS17}, we proved when $\mathfrak s$ is a simply type, but the generalization to all types follows from \cite{BK2} and a simple generalization of \cite[Theorem 2.1]{CS17}. The center of $\mathcal H_{\mathfrak s}$ is equal to $\mathcal Z_1 \otimes \ldots \otimes \mathcal Z_r$, where each $\mathcal Z_k$ is the center of $\mathcal H_{n_k}(q_k)$.

Recall that $\Pi_n=\mathrm{ind}_{U_n}^{G_n} \psi_n$. We may simply write $\Pi$ for $\Pi_n$ if there is no confusion.

\begin{theorem}\cite{CS17} \label{thm affine}
 For any $\mathfrak s \in \mathfrak B(G_n)$, the Bernstein component of the Gelfand-Graev representation $\Pi_{\mathfrak s}$ is isomorphic to $\mathcal H_{\mathfrak s} \otimes_{\mathcal H_{W,\mathfrak s}} \mathrm{sgn}_{\mathfrak s}$. 
\end{theorem}


\begin{lemma}\label{lem zero hom gg}
Let $\pi \in \mathrm{Alg}(G_n)$ be admissible. Then 
\[\mathrm{Hom}_{G_n}(\pi, \Pi)=0 . \]
\end{lemma}

\begin{proof}
It suffices to prove for $\pi \in \mathrm{Irr}(G_n)$. Let $\mathfrak{s} \in \mathfrak B(G_n)$ such that $\pi \in \mathfrak R_{\mathfrak s}(G_n)$. Now $\pi_{\mathfrak s}$ is an irreducible finite-dimensional $\mathcal H_{\mathfrak s}$-module, but there is no finite dimensional submodule for 
\[(\mathrm{ind}_{U_n}^{G_n}\psi_n)_{\mathfrak s}\cong \mathcal H_{\mathfrak s} \otimes_{\mathcal H_{W,\mathfrak s}} \mathrm{sgn}_{\mathfrak s}\] 
since there is no finite-dimensional submodule of $\mathcal A_{\mathfrak s}$ as $\mathcal A_{\mathfrak s}$-module. Hence the Hom space is zero.
\end{proof}

\begin{remark}
One can generalize Lemma \ref{lem zero hom gg} to a connected quasisplit reductive group $G$ with a non-compact center.\footnote{This is also pointed out by the referee. The author would like to thank the referee for that.} For an irreducible $\pi \in \mathfrak R_{\mathfrak s}(G)$, 
\[  \mathrm{Hom}_G(\pi, \Pi) \cong \mathrm{Hom}_G(\pi, \Pi_{\mathfrak s}) \cong \mathrm{Ext}^{d(\mathfrak s)}_G(\Pi_{\mathfrak s}, D(\pi) ),
\]
where $D$ is the Aubert-Schneider-Stuhler-Zelevinsky dual, and $d(\mathfrak s)$ is the cohomological dimension of the Bernstein block $\mathfrak R_{\mathfrak s}(G)$. The last isomorphism follows from \cite{NP20} (also see \cite{SS97, Ch16}). Since $G$ has non-compact center, $d(\mathfrak s) \geq 1$. Hence, last Ext is zero by the projectivity of $\Pi$ \cite{CS}. The admissible case follows from the irreducible case.
\end{remark}

\subsection{Inertial equivalence classes} \label{ss inertial equ class}

For the following discussions, see, for example, \cite{Be92} and \cite{BK2}. An inertial equivalence class $\mathfrak s$ of $G_n$ can be represented by a pair $[G_{m_1}\times \ldots \times G_{m_r}, \rho_1 \boxtimes \ldots \boxtimes \rho_r]$, where $m_1+\ldots +m_r=n$ and each $\rho_k$ is a cuspidal $G_{m_k}$-representation. Two pairs 
\[ [G_{m_1}\times \ldots \times G_{m_r}, \rho_1\boxtimes \ldots \boxtimes \rho_r], [G_{m_1'}\times \ldots \times G_{m_s'}, \rho_1'\boxtimes \ldots \boxtimes \rho_s'] \]
represent the same inertial equivalence class if and only if $r=s$ and there exists a permutation $\sigma \in S_r$ such that
\[ G_{m_1}=G_{m_{\sigma(1)}'}, \ldots, G_{m_r}=G_{m_{\sigma(r)}'} \]
and 
\[ \rho_1\cong \chi_1 \otimes \rho_{\sigma(1)}', \ldots , \rho_r \cong \chi_r \otimes \rho_{\sigma(r)}' 
\]
for some unramified character $\chi_k$ of $G_{m_k}$ ($k=1,\ldots, r$). 

Let 
\[ \mathfrak s_1=[G_{m_1}\times \ldots \times G_{m_r}, \rho_1\boxtimes \ldots \boxtimes \rho_r] \in \mathfrak B(G_{n_1}), \]
\[ \mathfrak s_2=[G_{m_1'}\times \ldots \times G_{m_s'}, \rho_1' \boxtimes \ldots \boxtimes \rho_s'] \in \mathfrak B(G_{n_2}) .
\]
Then $\pi_1 \times \pi_2$ lies in $\mathfrak R_{\mathfrak s}(G_{n_1+n_2})$, where 
\[ \mathfrak s=[G_{m_1}\times \ldots \times G_{m_r}\times G_{m_1'}\times \ldots \times G_{m_s'}, \rho_1\boxtimes \ldots \boxtimes \rho_r \boxtimes \rho_1'\boxtimes \ldots \boxtimes \rho_s']
\]

From this, one deduces the following two lemmas:
\begin{lemma} \label{lem unique bcomponent product}
Let $\pi_1 \in \mathfrak R_{\mathfrak s}(G_{n_1})$ for some $\mathfrak s\in \mathfrak B(G_{n_1})$ and let $\pi_2\in \mathrm{Alg}(G_{n_2})$. Fix $\mathfrak t \in \mathfrak B(G_{n_1+n_2})$ with $(\pi_1 \times \pi_2)_{\mathfrak t}\neq 0$. There exists a unique $\mathfrak s' \in \mathfrak B(G_{n_1+n_2})$ such that $\pi_1 \times (\pi_2)_{\mathfrak s'} \cong (\pi_1 \times \pi_2)_{\mathfrak t}$.
\end{lemma}

\begin{lemma} \label{lem product noncusp}
Let $\pi_1 \in \mathrm{Alg}(G_{n_1})$ and let $\pi_2 \in \mathrm{Alg}(G_{n_2})$ with $n_1, n_2 \neq 0$. Then any irreducible subquotient of $\pi_1 \times \pi_2$ is not cuspidal.
\end{lemma}

\subsection{Indecomposability of Gelfand-Graev representations}

\begin{proposition} \label{prop indecompose gg rep}
For any $\mathfrak{s} \in \mathfrak{B}(G_n)$, any two non-zero submodules $\pi_1, \pi_2$ in $(\mathrm{ind}_{U_n}^{G_n}\psi_n)_{\mathfrak s}$ has non-zero intersection. In particular, $(\mathrm{ind}_{U_n}^{G_n}\psi_n)_{\mathfrak s}$ is indecomposable.
\end{proposition}

\begin{proof}
By Theorem \ref{thm affine},
 \[\Pi_{\mathfrak s}|_{\mathcal A_{\mathfrak{s}}} \cong \mathcal A_{\mathfrak{s}} .\] Since $\mathcal A_{\mathfrak s}$ is commutative, any two non-zero $\mathcal A_{\mathfrak s}$-submodules of $\mathcal A_{\mathfrak{s}}$ have non-zero intersection. Hence any two non-zero submodules of $\Pi_{\mathfrak s}$ also have non-zero intersection. 
\end{proof}

\begin{remark}
We give another proof for the indecomposability of $\Pi_{\mathfrak s}$ as below, whose argument can be applied to other connected quasisplit reductive groups $G$. For any $\mathfrak s \in \mathfrak R(G)$, \cite{BH} showed that $\mathrm{End}_G(\Pi_{\mathfrak s})$ is isomorphic to the Bernstein center $\mathfrak Z_{\mathfrak s}$ of $\mathfrak R_{\mathfrak s}(G)$. Since $\mathfrak R_{\mathfrak s}(G)$ is an indecomposable category, we also have $\mathfrak Z_{\mathfrak s}$ is indecomposable as a $\mathfrak Z_{\mathfrak s}$-module. This implies that $\mathrm{End}_G(\Pi_{\mathfrak s})$ is indecomposable as a $\mathfrak Z_{\mathfrak s}$-module and hence $\Pi_{\mathfrak s}$ is indecomposable in $\mathfrak R_{\mathfrak s}(G)$.

\end{remark}

\subsection{Jacquet functors on Gelfand-Graev representations}

\begin{lemma} \label{lem restrict}
Let $P=LN$ be the parabolic subgroup containing upper triangular matrices and block-diagonal matrices $\mathrm{diag}(g_1,\ldots, g_r)$ with $g_k \in  G_{i_k}$, where $i_1+\ldots +i_r=n$. Then $(\Pi_n )_N \cong \Pi_{i_1} \boxtimes \ldots \boxtimes \Pi_{i_r}$.
\end{lemma}

\begin{proof}
Let $w$ be a permutation matrix in $G_n$. Then $w(N) \cap U_n$ contains a unipotent subgroup $\left\{ I_n +tu_{k,k+1} : t \in F \right\}$ for some $k$ if and only if $w(N) \not\subset U_n^-$. Here $u_{k,k+1}$ is a matrix with $(k,k+1)$-entry $1$ and other entries $0$. For any such $w$, it gives that $PwB$ is the same unique open orbit in $G_n$. Now the geometric lemma in \cite[Theorem 5.2]{BZ} gives the lemma.
\end{proof}

\subsection{Some lemmas}

For $\pi \in \mathrm{Irr}$, let $\mathrm{csupp}'(\pi)$ be the cuspidal support of $\pi$ but without multiplicities, and in particular, $\mathrm{csupp}'(\pi)$ is a set. For an admissible representation $\pi$ of $G_n$, let 
\[ \mathrm{csupp}'(\pi) =\bigcup_{\omega \in \mathrm{JH}(\pi)} \mathrm{csupp}'(\omega) .
\]

\begin{lemma} \label{lem product cusp support}
Let $\omega \in \mathrm{Alg}(G_{n_1})$ be admissible with $n_1 \neq 0$. Then, for any irreducible subquotient $\lambda$ of $\Pi_{n_2} \times \omega$, $\mathrm{csupp}'(\lambda) \cap \mathrm{csupp}'(\omega) \neq \emptyset$. Similarly, for any irreducible subquotient $\lambda$ of $\omega \times \Pi_{n_2}$, $\mathrm{csupp}'(\lambda) \cap\mathrm{csupp}'(\omega)\neq \emptyset$.
\end{lemma}

\begin{proof}

We only sketch the proof for the first statement. Suppose $\lambda$ is an irreducible subquotient of $\Pi_{n_2} \times \omega$. We consider that 
\[ \lambda \hookrightarrow \rho_1 \times \ldots \times \rho_r\]
for some cuspidal representations $\rho_1, \ldots, \rho_r$. The adjointness gives a Jacquet functor associated to a unipotent subgroup $N$:
\[  \lambda_N \rightarrow \rho_1 \boxtimes \ldots \boxtimes \rho_r .
\]
Using the geometric lemma on $(\Pi_{n_2}\times \omega)_N$ to obtain a filtration, we then use $\rho_1 \boxtimes \ldots \boxtimes \rho_r $ appears in a subquotient of $(\Pi_{n_2}\times \omega)_N$ (also see Lemma \ref{lem subquot}) to obtain that $\rho_i \in \mathrm{csupp}'(\omega)$ for some some $i$. 
\end{proof}

\begin{corollary} \label{cor submodule version of pre lem}
We use the same setting as in the previous lemma. Let $\pi$ be a submodule of $\Pi_{n_2}\times \omega$, and let $\pi'$ be a submodule of $\omega \times \Pi_{n_2}$. Then 
\begin{enumerate}
\item for any irreducible subquotient $\tau$ of $\pi$, $\mathrm{csupp}'(\tau) \cap \mathrm{csupp}'(\omega) \neq \emptyset$; and
\item for any irreducible subquotient $\tau'$ of $\pi'$, $\mathrm{csupp}'(\tau') \cap \mathrm{csupp}'(\omega)\neq \emptyset$.
\end{enumerate}

\end{corollary}

\subsection{Bernstein center} \label{ss bern center}

The Bernstein center of a category $\mathfrak R$ is defined as the endomorphism ring of the identity functor in $\mathfrak R$. Denote by $\mathfrak Z_n$ the Bernstein center of $\mathfrak R(G_n)$. For $\mathfrak s \in \mathfrak B(G_n)$, denote by $\mathfrak Z_{\mathfrak s}$ the Bernstein center of $\mathfrak R_{\mathfrak s}(G_n)$, which is a finitely-generated commutative algebra \cite[Theorem 2.13]{BD}. We have
\[  \mathfrak Z_n \cong \prod_{\mathfrak s \in \mathfrak B(G_n)} \mathfrak Z_{\mathfrak s} .
  \]
It follows from \cite[Corollaire 3.4]{BD} that for any maximal ideal $\mathcal J$ of $\mathfrak Z_n$ and any finitely-generated $\pi$ of $G_n$, $\pi/\mathcal J(\pi)$ is admissible. For more properties of the Bernstein center, see, e.g. \cite{BD}, \cite{BH}, \cite{Ha14}.
 
\begin{lemma} \label{lem annihilate pi cusp}
Let $\omega \in \mathrm{Alg}(G_{n_1})$ be admissible. Fix $\tau \in \mathrm{Irr}(G_{n_1+n_2})$ such that $\mathrm{csupp}'(\omega) \cap \mathrm{csupp}'(\tau) =\emptyset$. Let $\mathcal J$ be the maximal ideal in $\mathfrak Z_{n_1+n_2}$ which annihilates $\tau$. 
\begin{enumerate}
\item Let $\pi$ be a submodule of $\Pi_{n_2} \times \omega$. Then $\pi/\mathcal J(\pi)=0$.
\item Let $\pi'$ be a submodule of $\omega \times \Pi_{n_2}$. Then $\pi'/\mathcal J(\pi')=0$. 
\end{enumerate}
\end{lemma}

\begin{proof}
We only prove (1) and a proof for (2) is similar. Each Bernstein component of $\Pi$ is finitely generated \cite{BH} and so each Bernstein component of $ \pi$ is also finitely-generated \cite{BD}. Thus $\pi/\mathcal J (\pi)$ is admissible \cite{BD} and is annihilated by $\mathcal J$. If $\pi/\mathcal J(\pi)$ is non-zero, then $\pi$ has a subquotient with irreducible composition factors isomorphic to $\tau'$ with $\mathrm{csupp}(\tau')=\mathrm{csupp}(\tau)$ \cite[Proposition 2.11]{BD}. However, this contradicts Corollary \ref{cor submodule version of pre lem}(1).
\end{proof}

\subsection{Geometric lemma on $\omega \times \Pi_i$} \label{sec geo lem}

Let $\pi_1 \in \mathrm{Alg}(G_{n_1})$ and let $\pi_2 \in \mathrm{Alg}(G_{n_2})$. Consider $N_j \subset G_{n_1+n_2}$ for some $j$. The geometric lemma \cite{BZ} gives that $(\pi_1 \times \pi_2)_{N_j}$ admits a filtration:
\begin{align} \label{eqn geom filt 1}
  \omega_r \subset \omega_{r-1} \subset \ldots \subset \omega_0=(\pi_1 \times \pi_2)_{N_j} 
\end{align}
such that, for each $r-1 \geq k \geq 0$,
\[  \omega_{k}/\omega_{k+1} \cong  \mathrm{Ind}^{G_{n-j}\times G_{j}}_P((\pi_1)_{N_{k}} \boxtimes (\pi_2)_{N_{j-k}})^w ,
\]
where $P$ is the parabolic subgroup in $G_{n-j}\times G_j$ containing matrices $\mathrm{diag}(g_1, g_2, g_3, g_4)$ with $g_1 \in G_{n_1-k}$, $g_2 \in G_{n_2-j+k}$, $g_3 \in G_{k}$ and $g_4 \in G_{j-k}$, and upper triangular matrices, and $((\pi_1)_{N_{k}} \boxtimes (\pi_2)_{N_{j-k}})^w$ is a $P$-representation with underlying space $(\pi_1)_{N_{k}} \boxtimes (\pi_2)_{N_{j-k}}$ such that the action of the unipotent part is trivial and the action of the Levi part $G_{n_1-k}\times G_{n_2-j+k}\times G_k \times G_{j-k}$ is given by:
\[ \mathrm{diag}(g_1, g_2, g_3,g_4). v_1\otimes v_2 = (\mathrm{diag}(g_1,g_3).v_1) \otimes (\mathrm{diag}(g_2, g_4).v_2) ,
\]
where $v_1 \in (\pi_1)_{N_k}$ and $v_2 \in (\pi_2)_{N_{j-k}}$. Here we consider the space to be zero if $k> n_1$ or $j-k> n_2$.

\begin{lemma} \label{lem first filtration on product}
Let $\omega \in \mathrm{Alg}(G_{n-i})$ be admissible. Let $\tau \in \mathrm{Irr}(G_i)$ such that $\mathrm{csupp}'(\omega) \cap \mathrm{csupp}'(\tau)=\emptyset$. Let $\mathcal J'$ be the maximal ideal in $\mathfrak Z_i$ which annihilates $\tau$, and set $\mathcal J=\mathcal J'\otimes \mathfrak Z_{n-i} \subset \mathfrak Z_{i} \otimes \mathfrak Z_{n-i}$. Set $N=N_{n-i}$. Then $(\omega \times \Pi_i)_N$ admits, as $G_{i} \times G_{n-i}$-representations, a short exact sequence
\[  0 \rightarrow \kappa_1 \rightarrow (\omega \times \Pi_i)_N \rightarrow \kappa_2 \rightarrow 0
\]
such that 
\begin{enumerate}
\item $\kappa_2/\mathcal J(\kappa_2)=0$, and
\item  $\kappa_1/\mathcal J(\kappa_1)$ is admissible and is isomorphic to $\tau \boxtimes \omega$ for some $\tau \in \mathrm{Alg}(G_i)$.
\end{enumerate}

\end{lemma}

\begin{proof}
We explicate the filtration (\ref{eqn geom filt 1}) in this situation. The bottom layer takes the form 
\[ \kappa_1 :=  \Pi_i \boxtimes \omega .
\]
(Note that both $N_{n-i}$ on $\omega$ and $N_{0}$ on $\Pi_i$ are trivial groups.) By applying $\mathcal J$ on $\kappa_1$, we obtain (2).

 Now set $\kappa_2 =\pi/\kappa_1$. In order to prove $\kappa_2/\mathcal J(\kappa_2)=0$, we continue to consider the filtration (\ref{eqn geom filt 1}). Now $\kappa_2$ admits a filtration whose successive quotients take the form:
\[ \mathrm{Ind}^{G_{n-i}\times G_{i}}_P((\omega)_{N_{k}} \boxtimes (\Pi_i)_{N_{n-i-k}})^w
\]
for $k<n-i$. Since $\omega$ is admissible, $(\omega)_{N_k}$ admits a (finite) filtration of irreducible $G_{n-i-k}\times G_k$-representations. Thus we obtain a finite filtration of $\kappa_2$ whose successive quotients take the form
\begin{align} \label{eqn fine quotient}
 Q=(\tau' \times \Pi_{2i+k-n}) \boxtimes (\tau'' \times \Pi_{n-i-k}),
\end{align}
where $\tau' \in \mathrm{Irr}(G_{n-k-i})$ and $\tau'' \in \mathrm{Irr}(G_k)$ satisfying 
\[ \mathrm{csupp}'(\tau'), \mathrm{csupp}'(\tau'') \subset \mathrm{csupp}'(\omega).\] Note that $G_{n-k-i}$ is not the trivial group (i.e. $n-k-i \neq 0$) and $\mathrm{csupp}'(\tau')\cap \mathrm{csupp}'(\tau)=\emptyset$, and so by Lemma \ref{lem annihilate pi cusp}, $(\tau'\times \Pi_{2i+k-n})/\mathcal J'(\tau'\times \Pi_{2i+k-n})=0$. Thus, successive quotients in (\ref{eqn fine quotient}) give 
\[   Q/\mathcal J(Q)=0.
\]


Thus we have proved that $\kappa_2$ admits a filtration
\[  0=F_0 \subset F_1 \subset \ldots \subset F_l =\kappa_2
\]
such that $Q_k/\mathcal J(Q_k)=0$, where $Q_k =F_k/F_{k-1}$ for 
$k=1, \ldots, l$. Now the short exact sequence
\[   0 \rightarrow F_{k-1} \rightarrow F_k \rightarrow Q_k \rightarrow 0
\]
gives a right exact sequence
\[  F_{k-1}/\mathcal J(F_{k-1}) \rightarrow F_k/\mathcal J(F_k) \rightarrow Q_k/\mathcal J(Q_k) \rightarrow 0
\]
Thus inductively, we have $F_k/\mathcal J(F_k)=0$ for all $k$, including $k=l$. This proves (1).

\end{proof}


\begin{lemma} \label{lem short exact z invariant}
Keep the same setting as the previous lemma. Let $\pi$ be a non-zero submodule of $\omega \times \Pi_i$. Then $\pi_N$, as a $G_i \times G_{n-i}$-representation, admits a short exact sequence:
\[ 0 \rightarrow \kappa_1' \rightarrow \pi_N \rightarrow \kappa_2' \rightarrow 0
\]
such that $\kappa_2'/\mathcal J(\kappa_2')=0$ and moreover, $\kappa_1'$ is a submodule of $\Pi_i \boxtimes \omega$.
\end{lemma}

\begin{proof}
We use the notations in the proof of Lemma \ref{lem first filtration on product}. Thus we have a short exact sequence for $(\omega \times \Pi_i)_N$. Hence, intersecting with $\pi$, we have a short exact sequence for $\pi_N$:
\[   0 \rightarrow \kappa_1' \rightarrow \pi_N \rightarrow \kappa_2' \rightarrow 0 .
\]

We can then further obtain a fine filtration on $\kappa_2'$ with each successive quotient $Q'$ satisfying:
\[  Q' \hookrightarrow  (\tau' \times \Pi_{2i+k-n}) \boxtimes (\tau'' \times \Pi_{n-i-k})
\]
(see (\ref{eqn fine quotient})). Since $\mathrm{csupp}'(\tau)\cap \mathrm{csupp}'(\tau')=\emptyset$, one can follow the argument in Lemmas \ref{lem product cusp support} and \ref{lem annihilate pi cusp} to conclude that $Q'/\mathcal J(Q')=0$.

\end{proof}

We also derive another consequences, and later on we will remove an assumption of the following lemma in Corollary \ref{cor intersection}. 

\begin{lemma} \label{lem admit quotient}
Keep the same setting as the previous two lemmas. Again let $\pi$ be a non-zero submodule of $\omega \times \Pi_i$. Suppose $\pi_N/\mathcal J(\pi_N) \neq 0$. Then there exists a non-zero submodule $\tau$ of $\pi_N$ such that 
\begin{enumerate}
\item $\tau \hookrightarrow \Pi_i \boxtimes \omega$, and 
\item regard $\tau$ as $G_{n-i}$-representation via the embedding $g \mapsto \mathrm{diag}(I_{i},g)$, 
\[ \mathrm{Hom}_{G_{n-i}}(\omega', \tau) \neq 0 .
\]
for some submodule $\omega'$ of $\omega$.
\end{enumerate}
\end{lemma}

\begin{proof}
We keep using the notations in the proof of the previous lemma. Note that (2) follows from (1). For (1), it suffices to show that $\kappa_1'$ in Lemma \ref{lem short exact z invariant} is non-zero. Indeed, we have the following right exact sequence:
\[  (\kappa_1')/\mathcal J(\kappa_1') \rightarrow \pi_N/\mathcal J(\pi_N) \rightarrow (\kappa_2')/\mathcal J(\kappa_2') \rightarrow 0 .
\]
The last term is zero by Lemma \ref{lem short exact z invariant}. Now $\pi_N/\mathcal J(\pi_N)\neq 0$ implies $(\kappa_1')/\mathcal J(\kappa_1') \neq 0$. Hence $\kappa_1' \neq 0$.
\end{proof}

\subsection{Geometric lemma on $\Pi_i \times \omega$}

We can obtain a parallel version of Lemma \ref{lem admit quotient} for $(\Pi_i \times \omega)_{N_{n-i}}$. Since we do not really need that, we only state without a proof. 
\begin{lemma} \label{lem adimit quotient 2}
Let $\omega \in \mathrm{Alg}(G_{n-i})$ be admissible. Let $\tau \in \mathrm{Irr}(G_i)$ such that $\mathrm{csupp}'(\omega) \cap \mathrm{csupp}'(\tau)=\emptyset$. Let $\mathcal J'$ be the maximal ideal in $\mathfrak Z_i$ which annihilates $\tau$, and set $\mathcal J=\mathcal J'\otimes \mathfrak Z_{n-i} \in \mathfrak Z_{i} \otimes \mathfrak Z_{n-i}$. Set $N=N_{n-i}$. Then $(\Pi_i \times \omega)_N$ admits a short exact sequence:
\[  0 \rightarrow \eta_1 \rightarrow (\Pi_i \times \omega)_N \rightarrow \eta_2 \rightarrow 0
\]
such that $\eta_1/\mathcal J(\eta_1)=0$. Moreover, suppose $\pi$ is a submodule of $\Pi_i \times \omega$. If $\pi_N/\mathcal J(\pi_N)\neq 0$, then $\pi \cap \eta_1 \neq \pi$. 
\end{lemma}

We now prove a lemma that is needed later:

\begin{lemma} \label{lem second geom lemma}
Let $\omega \in \mathrm{Alg}(G_{n-i})$ be admissible. Let $\pi$ be a submodule of $\Pi_i \times \omega$. Set $N=N_{n-i}$. Then there exists a short exact sequence on $\pi_N$:
\[  0 \rightarrow \eta_1  \rightarrow \pi_N \rightarrow \eta_2 \rightarrow 0
\]
such that, as $G_i \times G_{n-i}$-representations,
\begin{enumerate}
\item[(1)] $\eta_2$ embeds to $\Pi_i \boxtimes \omega$;
\item[(2)] regard as $G_{n-i}$-representations via the embedding $g \mapsto \mathrm{diag}(I_i, g)$, $\eta_1$ does not have any admissible $G_{n-i}$-submodule;
\end{enumerate}
In particular, regard $\pi_N$ as $G_{n-i}$-representation, if $\tau\in \mathrm{Irr}(G_{n-i})$ embeds $\pi_N$, then $\tau$ is an irreducible submodule of $\omega$.
\end{lemma}

\begin{proof}
Again we consider the filtration from the geometric lemma (\ref{eqn geom filt 1}). We refine the filtration by incorporating the the filtrations on Jacquet functors of $\omega$ by simple composition factors. We also use Lemma \ref{lem restrict}. Then we obtain a filtration on $(\Pi_i \times \omega)_N$:
\[  0 \subset \eta_r' \subset \ldots \subset \eta_0'=(\Pi_i \times \omega)_N
\]
such that $\eta_0'/\eta_{1}' \cong \Pi_i \boxtimes \omega$ and for $k>1$,
\begin{align}\label{eqn iso quotient gg} \eta_{k-1}'/ \eta_{k}' \cong (\Pi_{i-l} \times \omega') \boxtimes (\Pi_l \times \omega'')
\end{align}
for some $l \geq 1$ and $\omega' \in \mathrm{Irr}(G_l)$ and $\omega'' \in \mathrm{Irr}(G_{n-i-l})$.

We set $\eta_1=\eta_1'\cap \pi_N$ and $\eta_2= \pi_N/(\eta_1'\cap \pi_N)$. It remains to check that the short exact sequence
\[ 0 \rightarrow \eta_1 \rightarrow \pi_N \rightarrow \eta_2 \rightarrow 0
\]
satisfies the required properties. (1) is immediate from the definition of $\eta_2$.

Now we consider (2). Suppose $\tau \in \mathrm{Irr}(G_{n-i})$ such that 
\[  \mathrm{Hom}_{G_{n-i}}(\tau, \eta_1 ) \neq 0 
\]
and we shall arrive a contradiction. Then
\[ \mathrm{Hom}_{G_{n-i}}(\tau, (\pi_N \cap \eta_{k-1}')/(\pi_N\cap \eta_k')) \neq 0
\]
for some $k\geq 2$. Since 
\[  (\pi_N \cap \eta_{k-1}')/(\pi_N\cap \eta_k') \hookrightarrow \eta_{k-1}'/\eta_k' ,
\]
we have that, by the expression (\ref{eqn iso quotient gg}),
\[  \mathrm{Hom}_{G_{n-i}}(\tau, \Pi_l\times \omega'') \neq 0
\]
for some $\omega'' \in \mathrm{Alg}(G_{n-i-l})$. However, the last Hom is zero as $l>0$ (see Lemma \ref{lem zero hom gg} as well as the proof of Lemma \ref{lem bz filtration} below). This gives a contradiction and thus we have (2).
\end{proof}

\subsection{Intersection properties of Bernstein-Zelevinsky layers}


The main tool of this section is the exactness of Jacquet functor and the existence of certain irreducible subquotient in Corollary \ref{lem exist quotient}, which reduces the problem of studying submodules of a representation to certain submodules of the Jacquet module. The key property needed for Gelfand-Graev representation is Lemma \ref{lem compos factor 1}, in which such construction is most conveniently done in the setting of affine Hecke algebra.

The following result can be viewed as a generalization of the fact that the Gelfand-Graev representation does not admit an admissible submodule.

\begin{lemma} \label{lem bz filtration}
For $i\neq j$, let $\omega, \tau$ be admissible $G_{n-i}$-representation and $G_{n-j}$-representation respectively. There is no non-zero isomorphic  submodule of $\mathrm{ind}^{G_n}_{R_{n-i}} \omega \boxtimes \psi_i$ and $\mathrm{ind}_{R_{n-j}^-}^{G_n} \tau  \boxtimes \psi_j$.
\end{lemma}

\begin{proof}

We only prove for $i>j$ and the case for $j<i$ is similar. Then we have that $(\mathrm{ind}^{G_n}_{R_{n-i}} \omega \otimes \psi_i)_{N_i}$ admits a filtration that successive quotients take the form 
\[  ( \omega' \times \Pi_k) \boxtimes (\omega'' \times \Pi_{i-k} ).
\]
Similarly,  $(\mathrm{ind}^{G_n}_{R_{n-i}^-} \tau \boxtimes \psi_j)_{N_i}$ admits a filtration that successive quotients take the form
\[  (\Pi_l \times \tau')  \boxtimes (\Pi_{j-l}\times  \tau''),
\]
where $\tau'$ is an admissible $G_{n-i-l}$-representation and $\tau''$ is an admissible $G_{i-j+l}$ representation.
Since Jacquet functor is exact, it reduces, by Lemma \ref{lem two filtration intersect}, to see that there is no subrepresentation for 
\[  (\Pi_l \times \tau')  \boxtimes (\Pi_{j-l}\times  \tau''), \quad \mbox{ and }  ( \omega' \times \Pi_k) \boxtimes (\omega'' \times \Pi_{i-k} ).
\]
for all $k,l$.

In the case that $k=0$ and $l=j$; or $k=n-i$ and $l=0$, a certain element in the Bernstein center annihilates $\omega''$ or $\omega'$, but not on $\Pi_i$ or $\Pi_{n-i}$ respectively (by \cite{BH}). Now to prove other cases, we notice that either $k>l$ or $i-k>j-l$ (which happens when $l\geq k$). In the first case, we apply $N_k \subset G_{n-i}$ on the first factor, and in the second case, apply $N_{i-k}$ on the second factor. Then repeat the process and the process terminates as each process the value on $i$ decreases.

\end{proof}

\begin{lemma} \label{lem compos factor 1}
Let $\mathfrak s \in \mathfrak B(G_n)$. Let $\tau=\langle \mathfrak m \rangle =\zeta(\mathfrak m)$ be an irreducible principle series in $\mathfrak R_{\mathfrak s}(G_n)$, where $\mathfrak m \in \mathrm{Mult}$ with all segments to be singletons. Then there exists a $G_n$-filtration on $\Pi_{\mathfrak s}$:
\[                    \ldots     \subsetneq    \pi_2  \subsetneq  \pi_1  \subsetneq \pi_0=\Pi_{\mathfrak s}
\]
such that 
\begin{enumerate}
\item each successive quotient $\pi_{r-1}/\pi_{r}$  is admissible and 
\[\mathrm{JH}(\pi_{r-1}/\pi_r)=\left\{\tau\right\}= \left\{\langle \mathfrak m\rangle \right\}\] and;
\item 
\[  \bigcap_{r\geq 0} \pi_r = 0 .
\]
\end{enumerate}
\end{lemma}

\begin{proof}

We shall prove this by the Hecke algebra model of the Gelfand-Graev representation in Section \ref{ss affine hecke}, and we shall use those notations. Let $\mathcal Z$ be the center of $\mathcal H_{\mathfrak s}$ and let $\mathcal J$ be the maximal ideal in $\mathcal Z$ annihilating $\tau_{\mathfrak s}$. By Theorem \ref{thm affine}, 
\[ \Pi_{\mathfrak s} \cong \mathcal H_{\mathfrak s}\otimes_{\mathcal H_{W,\mathfrak s}} \mathrm{sgn} ,\]
We consider the filtration given by $\pi_k=\mathcal J^k\Pi\cong \mathcal J^k\mathcal H_{\mathfrak s}\otimes_{\mathcal H_{W,\mathfrak s}}\mathrm{sgn}$.  Since $\mathcal H_{\mathfrak s}$ is finitely-generated as a $\mathcal Z$-module, each successive quotient of the filtration from $\pi_k$ is finite-dimensional. With the fact that $\tau_{\mathfrak s}$ is the only irreducible $\mathcal H_{\mathfrak s}$-module annihilated by $\mathcal J$, we have that any irreducible composition factor of $\pi_k/\pi_{k-1}$ is isomorphic to $\tau_{\mathfrak s}$. This gives the corresponding property (1) in the group side. For (2), we have that $\pi_k|_{\mathcal A_{\mathfrak s}} \cong \mathcal J^k\mathcal A_{\mathfrak s}$. Now $\bigcap_k \mathcal J^k =0$ and $\mathcal A_{\mathfrak s}=\mathcal Za_1+\ldots +\mathcal Za_k$ for finitely many $a_1, \ldots a_k \in \mathcal A_{\mathfrak s}$, and this gives that $\bigcap_r \pi_r =0$ as desired. 

\end{proof}



\begin{lemma} \label{lem exist quotient}
Let $\mathfrak s \in \mathfrak{B}(G_{n-i})$ and let $\mathfrak t \in \mathfrak B(G_i)$. Let $\omega$ be an irreducible representation in $\mathfrak R_{\mathfrak s}(G_{n-i})$. Let $\pi$ be a non-zero submodule of $\omega \times (\Pi_i)_{\mathfrak t}$. Let $\tau \in \mathfrak R_{\mathfrak t}(G_i)$ be an irreducible principle series as in Lemma \ref{lem compos factor 1} and further assume that $\mathrm{csupp}_{\mathbb{Z}}(\tau) \cap \mathrm{csupp}_{\mathbb Z}(\omega)=\emptyset$. Then there exists a submodule $\pi'$ of $\omega \times (\Pi_i)_{\mathfrak t}$ such that 
\begin{enumerate}
\item the projection from $\pi$ to $(\omega \times (\Pi_i)_{\mathfrak t})/\pi'$ is non-zero; and 
\item  the image $\pi/(\pi' \cap \pi)$ of the projection is admissible and any irreducible subquotient of the image is isomorphic to $\omega \times \tau$.
\end{enumerate}
\end{lemma}

\begin{remark} \label{rmk lem product quoteint z}
\begin{enumerate}
\item  The condition that $\mathrm{csupp}_{\mathbb{Z}}(\tau) \cap \mathrm{csupp}_{\mathbb Z}(\omega)=\emptyset$ implies that $\tau \times \omega$ is irreducible. 
\item One can replace an irreducible representation $\omega$ with an admissible representation $\omega$ in $\mathfrak R_{\mathfrak s}(G_{n-i})$. Then one replaces $\mathrm{csupp}_{\mathbb Z}(\omega)$ by $\cup \mathrm{csupp}_{\mathbb Z}(\omega')$, where the union runs for all $\omega' \in \mathrm{JH}(\omega)$; and replace the last sentence ''irreducible subquotient isomorphic to $\omega \times \tau$'' with ''irreducible subquotient isomorphic to $\omega' \times \tau$ for some $\omega' \in \mathrm{JH}(\omega)$''. In order to prove such statement, one considers a filtration on the admissible representation $\omega$:
\[  0 \subset \tau_1 \subset \ldots \subset \tau_r=\omega ,
\]
with irreducible successive quotients. This induces a filtration on $\omega \times (\Pi_i)_{\mathfrak t}$. Intersecting $\pi$ with $\tau_k \times (\Pi_i)_{\mathfrak t}$, we obtain a filtration for $\pi$. Now we consider the largest $k$ such that $\pi \cap (\tau_k \times (\Pi_i)_{\mathfrak t}) \neq \pi \cap (\tau_{k-1}\times (\Pi_i)_{\mathfrak t})$, and obtain
\[ 0\neq \pi \cap (\tau_k \times (\Pi_i)_{\mathfrak t})/\pi \cap (\tau_{k-1}\times (\Pi_i)_{\mathfrak t}) \hookrightarrow (\tau_k/\tau_{k-1}) \times  (\Pi_i)_{\mathfrak t} .
\]
This embedding reduces to the original Lemma \ref{lem exist quotient} and one can obtain the required submodule. 
\item There are only countably many elements in $\mathrm{csupp}_{\mathbb Z}(\omega)$ and so such $\tau$ always exists. 
\end{enumerate}
\end{remark}

\begin{proof}
 Let 
\[    \ldots \subset   \lambda_2' \subset \lambda_1'  \subset \lambda_0' =(\Pi_i)_{\mathfrak t}
\]
be a filtration of $(\Pi_i)_{\mathfrak t}$ with the properties in Lemma \ref{lem compos factor 1}. Then we have a filtration
\[   \ldots \subset \omega \times \lambda_2' \subset \omega \times \lambda_1' \subset \omega \times \lambda_0' =\omega \times (\Pi_i)_{\mathfrak s}
\]
whose successive subquotients, by exactness of parabolic induction, are admissible with composition factors isomorphic to $\omega \times \tau$. Since 
\[ \bigcap_k (\omega \times \lambda_k') \subset \left(\omega \times \bigcap_k \lambda_k' \right)=0 ,\]
we must have that the projection of $\pi$ to $(\omega \times (\Pi_i)_{\mathfrak s})/(\omega \times \lambda_k')$ is non-zero for sufficiently large $k$. Take such $\pi'=\omega \times \lambda_k'$. Hence, we obtain property (1). Now (2) follows from (1) and the description of the filtration above since $\pi/(\pi'\cap \pi)$ embeds to $(\omega \times (\Pi_i)_{\mathfrak t})/\pi'$. 
\end{proof}


We now can show the existence of certain irreducible $(1 \times G_{n-i})$-submodule in $(\omega \times \Pi_i)_{N_{n-i}}$. 

\begin{corollary} \label{cor intersection}
Let $0 \neq \omega \in \mathrm{Alg}(G_{n-i})$ be admissible. Set $N=N_{n-i}$. Let $\pi$ be a non-zero submodule of $\omega \times \Pi_i$. Regard $\pi_N$ as a $G_{n-i}$-representation via the embedding $g \mapsto \mathrm{diag}(I_i,g)$, 
\[  \mathrm{Hom}_{G_{n-i}}(\omega', \pi_N) \neq 0 
\]
for some irreducible submodule $\omega'$ of $\omega$.
\end{corollary}

\begin{proof}
In view of Lemma \ref{lem admit quotient}, we have to find an ideal $\mathcal J$ in the Bernstein center such that $\pi_N/\mathcal J(\pi_N) \neq 0$ and some additional properties. To this end, we write 
\[ \omega =\omega_1 \oplus \ldots \oplus \omega_r \]
 such that each $\omega_k \in \mathfrak R_{\mathfrak s_k}(G_{n-i})$ for some $\mathfrak s_k$ and $\mathfrak s_k \neq \mathfrak s_l$ if $k \neq l$. Then 
\[  \omega \times \Pi_i = \omega_1 \times \Pi_i \oplus \ldots \oplus \omega_r \times \Pi_i .
\]
Thus, by using lemma \ref{lem reduction}, there exists a non-zero submodule $\widetilde{\pi}$ of $\pi$ such that
\[ \widetilde{\pi} \hookrightarrow \omega_{k^*} \times \Pi_i \]
 for some $k^*$. By projecting $\widetilde{\pi}$ to a Bernstein component if necessary, we also assume that $\widetilde{\pi}$ lies in $\mathfrak R_{\mathfrak u}(G_{n})$ for some $\mathfrak u \in \mathfrak B(G_{n})$.  Now $\mathfrak u$ and $\mathfrak s_{k^*}$ also determine a $\mathfrak t \in \mathfrak B(G_i)$ such that 
\[(\omega_{k^*} \times \Pi_i)_{\mathfrak u}=\omega_{k^*} \times (\Pi_i)_{\mathfrak t} \]
 (see Lemma \ref{lem unique bcomponent product}). Now we can choose $\tau$ satisfying assumptions in Lemma \ref{lem compos factor 1} (see Remark \ref{rmk lem product quoteint z} (2) and (3)). Thus we also have the condition that 
\[ \mathrm{csupp}'(\tau) \cap \mathrm{csupp}'(\omega_{k^*}) = \emptyset .\] Now we set $\mathcal J' \subset  \mathfrak Z_i$ the maximal ideal annihilating $\tau$, and set $\mathcal J=\mathcal J'\boxtimes \mathfrak Z_{n-i}$, an element in $\mathfrak Z_{i}\otimes \mathfrak Z_{n-i}$. Now, by Lemma \ref{lem exist quotient}, there is a surjection:
\[ \widetilde{\pi} \rightarrow  \tau \times \omega \rightarrow 0
\]
and so a surjection:
\[ \widetilde{\pi}_N \rightarrow (\tau \times \omega)_N \rightarrow 0 .
\]
Since $(\tau \times \omega)_N/\mathcal J((\tau \times \omega)_N) \neq 0$, we obtain $\widetilde{\pi}_N/\mathcal J(\widetilde{\pi}_N)$. Now, by Lemma \ref{lem admit quotient}, we obtain 
\[ 0\neq  \mathrm{Hom}_{G_{n-i}}(\omega', \widetilde{\pi}) \hookrightarrow \mathrm{Hom}_{G_{n-i}}(\omega', \pi_N)
\]
for some irreducible submodule $\omega'$ of $\omega_{k^*}$. This proves the corollary.
\end{proof}

\begin{proposition} \label{prop intersect same}
Let $ \omega_1, \omega_2 \in \mathrm{Alg}(G_{n-i})$ be admissible and non-zero. If $\mathrm{ind}_{R_{n-i}}^{G_n} \omega_1 \boxtimes \psi_i$ and $\mathrm{ind}_{R_{n-i}^-}^{G_n} \omega_2 \boxtimes \psi_i$ have isomorphic non-zero submodules, then $\omega_1$ and $\omega_2$ have isomorphic non-zero submodules.
\end{proposition}

\begin{proof}


Let $0 \neq \pi \in \mathrm{Alg}(G_n)$ such that 
\[  \pi \stackrel{\iota_1}{\hookrightarrow}  \omega_1 \times \Pi_i, \quad \mbox{ and} \quad \pi \stackrel{\iota_2}{\hookrightarrow} \Pi_i \times \omega_2 .\]
By Corollary \ref{cor intersection}, there exists an irreducible submodule $\omega'$ of $\omega_1$ such that
\[  \mathrm{Hom}_{G_{n-i}}(\omega', \pi_N)\neq 0 .
\]
By Lemma \ref{lem second geom lemma}, $\omega'$ is an irreducible submodule of $\omega_2$. This proves the proposition.

\end{proof}




\subsection{Strong indecomposability}

A $G_n$-representation $\pi$ is said to be strongly indecomposable if any two non-zero submodules of $\pi$ have non-zero intersection.

We now prove a main property of a Bernstein-Zelevinsky layer. In Section \ref{s preserve indecomp}, we shall prove a variation, which says that the Bernstein-Zelevinsky induction preserves indecomposability.

\begin{lemma} \label{lem product strong indecomp}
Let $\pi_1 \in \mathrm{Alg}(G_{n_1})$ and $\pi_2 \in \mathrm{Alg}(G_{n_2})$. Suppose $\pi_1$ is strongly indecomposable and $\pi_2$ is irreducible. Then $\pi_1 \boxtimes \pi_2$ is also strongly indecomposable, as a $G_{n_1}\times G_{n_2}$-representation.
\end{lemma}

\begin{proof}
This follows from that any submodule of $\pi_1\boxtimes \pi_2$ is the subspace $\tau \boxtimes \pi_2 \subset \pi_1 \boxtimes \pi_2$ for some submodule $\tau$ of $\pi_1$.
\end{proof}

\begin{lemma} \label{lem jacquet mod inclusion}
Let $\pi_1, \pi_2, \pi$ be in $\mathrm{Alg}(G_n)$ such that $\pi_1 \hookrightarrow \pi$ and $\pi_2 \hookrightarrow \pi$. Then $(\pi_1 \cap \pi_2)_N \cong (\pi_1)_N \cap (\pi_2)_N$. Here the later intersection is taken in $\pi_N$. 
\end{lemma}

\begin{proof}
We have the natural projection $p:\pi_1 \cap \pi_2 \rightarrow \pi_N$ as linear spaces. Since the image of the projection lies in both $(\pi_1)_N$ and $(\pi_2)_N$, the projection factors through the embedding $(\pi_1)_N \cap (\pi_2)_N$ to $\pi_N$. We have the following commutative diagram:
\[ \xymatrix{ \pi_1 \cap \pi_2 \ar[r]^p \ar[dr] & \pi_N \\ & (\pi_1)_N \cap (\pi_2)_N \ar@{^{(}->}[u] \\   }\] 
Now taking the Jacquet functor in the diagram, $p$ induces an isomorphism from $(\pi_1\cap \pi_2)_N$ onto $(\pi_1\cap \pi_2)_N \subset \pi_N$. Thus the map from $(\pi_1)_N \cap (\pi_2)_N$ to $(\pi_1 \cap \pi_2)_N$ is surjective. Hence we obtain an isomorphism between $(\pi_1)_N \cap (\pi_2)_N$ and $(\pi_1 \cap \pi_2)_N$.
\end{proof}

\begin{theorem} \label{thm strong indecomp}
Let $\omega \in \mathrm{Irr}(G_{n-i})$. For any $\mathfrak s \in \mathfrak B(G_n)$, $(\mathrm{ind}_{R_{n-i}}^{G_n} \omega \boxtimes \psi_i)_{\mathfrak s}$ is strongly indecomposable whenever it is non-zero.
\end{theorem}

\begin{proof}

 Let $\pi_1$ and $\pi_2$ be non-zero submodules of $(\omega \times \Pi_i)_{\mathfrak s}$. Thus we have injections:
\[  \pi_1 \stackrel{\iota_1}{\hookrightarrow} ( \omega \times \Pi_i)_{\mathfrak s}\subset \omega \times \Pi_i \quad \mbox{ and } \quad \pi_2  \stackrel{\iota_2}{\hookrightarrow} (\omega \times \Pi_i)_{\mathfrak s} \subset \omega \times \Pi_i .\]
Following the same argument as in Corollary \ref{cor intersection} (which uses Lemma \ref{lem compos factor 1}), for $j=1,2$, we have a non-zero submodule $\tau_j$ of $(\pi_j)_N$ such that 
\[\tau_j \hookrightarrow \Pi_i \boxtimes \omega.\]
The Bernstein component $\mathfrak R_{\mathfrak s}(G_n)$ and the Bernstein component for $\omega$ determines that 
\[\tau_j \hookrightarrow (\Pi_i)_{\mathfrak t} \boxtimes \omega \]
 for some $\mathfrak t \in \mathfrak B(G_i)$ (c.f. Lemma \ref{lem unique bcomponent product}). Now, as $(\Pi_i)_{\mathfrak t} \boxtimes \omega$ is strongly indecomposable (Lemma \ref{lem product strong indecomp} and Proposition \ref{prop indecompose gg rep}), we conclude 
\begin{align} \label{eqn non-zero intersect bottom}
\tau_1 \cap \tau_2 \neq 0 .
\end{align}

 To show $\pi_1 \cap \pi_2 \neq 0$, it suffices to show that $(\pi_1 \cap \pi_2)_N \neq 0$ as the Jacquet functor is exact. By Lemma \ref{lem jacquet mod inclusion}, it is equivalent to show that $(\pi_1)_N \cap (\pi_2)_N \neq 0$, which indeed follows from (\ref{eqn non-zero intersect bottom}):
\[ 0 \neq \tau_1 \cap \tau_2 \subset (\pi_1)_N \cap (\pi_2)_N . \]
\end{proof}



\section{Indecomposability of restricted representations} \label{sec indecomp}

\subsection{Indecomposability of restriction } \label{ss indecomposable}

We now prove our main result, and the proof is in a similar spirit to that of Theorem \ref{thm strong indecomp}:

\begin{theorem} \label{thm indecomp irred}
Let $\pi \in \mathrm{Irr}(G_{n+1})$. Then for each $\mathfrak s \in \mathfrak R(G_n)$, $\pi_{\mathfrak s}$ is strongly indecomposable whenever it is nonzero i.e. for any two non-zero submodules $\tau, \tau'$ of $\pi_{\mathfrak s}$, $\tau \cap \tau' \neq 0$.
\end{theorem}

\begin{proof}

There exists a Bernstein-Zelevinsky $G_n$-filtration on $\pi$ with
\begin{align} \label{eqn bz filtration}
    \pi_n  \subset \pi_{n-1} \subset \ldots \subset \pi_1 \subset \pi_0 =\pi
\end{align}
such that 
\begin{align} \label{eqn bz layer r}   \pi_{i}/\pi_{i+1} \cong \mathrm{ind}_{R_{n-i}}^{G_n}  \pi^{[i+1]} \boxtimes \psi_{i} .
\end{align}
We also have a $G_n$-filtration on $\pi$ with
\begin{align} \label{eqn bz filtration l}
   {}_n\pi \subset {}_{n-1} \pi \subset \ldots \subset {}_1\pi \subset {}_0\pi =\pi 
	\end{align}
such that
\begin{align} \label{eqn bz layer l}
   {}_{i}\pi/{}_{i+1}\pi \cong \mathrm{ind}_{R_{n-i}^-}^{G_n}  {}^{[i+1]}\pi \boxtimes \psi_{i}.
	\end{align}

Let $i^*$ be the level of $\pi$. For $\mathfrak s \in \mathfrak B(G_n)$ such that $\pi_{\mathfrak s} \neq 0$, we also have $(\pi_{i^*-1})_{\mathfrak s} \neq 0$ (see Section \ref{ss remark cuspidal rep} for the detail). For notation simplicity, we set $\tau =(\pi_{i^*-1})_{\mathfrak s}$ and $\tau'=({}_{i^*-1}\pi)_{\mathfrak s}$, both regarded as subspaces of $\pi$.

Let $\omega$ and $\gamma$ be two non-zero submodules of $\pi_{\mathfrak s}$. \\

\noindent
{\it Claim:} $\omega \cap \tau \neq 0$. \\
{\it Proof of the claim:} Suppose not. Then, the natural projection gives an injection 
\begin{align}\label{eqn embedding}   \omega \hookrightarrow  \pi_{\mathfrak s}/\tau
\end{align}

By Lemma \ref{lem bz filtration}, there is no isomorphic submodules between $\pi_{\mathfrak s}/\tau$ and $\tau'$. This implies $\omega \cap \tau'=0$. Hence, we also have an injection:
\begin{align}\label{eqn embedding 2}   \omega \hookrightarrow  \pi_{\mathfrak s}/\tau'
\end{align}


By (\ref{eqn bz filtration}), (\ref{eqn bz filtration l}), (\ref{eqn embedding}), (\ref{eqn embedding 2}) and Lemma \ref{lem two filtration intersect}, there exists a $G_n$-representation which is isomorphic to submodules of 
\[  (  \pi^{[j]} \times \Pi_{j-1})_{\mathfrak s}, \quad \mbox{ and } \quad  (\Pi_{k-1} \times {}^{[k]}\pi )_{\mathfrak s}
\]
for some $j,k< i^*$. By Lemma \ref{lem bz filtration} again, we must have $j=k$. However, Proposition \ref{prop intersect same} contradicts to the following Theorem \ref{lem same quo derivative} below (whose proof does not depend on this result). This proves the claim. \\

Since $\omega$ is an arbitrary submodule of $\pi$, we also have $\gamma \cap \tau \neq 0$. Since $\pi^{(i^*)}$ is irreducible, $\tau$ is strongly indecomposable by Theorem \ref{thm strong indecomp}. This implies $\omega \cap \gamma \neq 0$.

\end{proof}

\begin{theorem} \label{lem same quo derivative}
Let $\pi \in \mathrm{Irr}(G_{n+1})$. If $i$ is not the level for $\pi$, then $\pi^{[i]}$ and ${}^{[i]}\pi$ do not have an isomorphic irreducible quotient, and also do not have an isomorphic irreducible submodule whenever the two derivatives are not zero.

\end{theorem}

The proof of Theorem \ref{lem same quo derivative} will be carried out in Section \ref{sec proof lemma}. Note that the converse of the above theorem is also true, which follows directly from the well-known highest derivative due to Zelevinsky \cite[Theorem 8.1]{Ze}.

\subsection{Non-zero Bernstein components} \label{ss remark cuspidal rep}
Let $\pi \in \mathrm{Irr}(G_{n+1})$. In order to give an explicit parametrization of indecomposable components of $\pi|_{G_n}$, we also have to determine when $\pi_{\mathfrak s}\neq 0$ for $\mathfrak s \in \mathfrak B(G_n)$. Indeed this can be done as follows. Write $\pi \cong \langle \mathfrak m \rangle $ for a multisegment $\mathfrak m=\left\{ \Delta_1, \ldots, \Delta_k \right\}$. Let $\pi'$ be the (right) highest derivative of $\pi$. Then we obtain a multiset 
\[\mathrm{csupp}(\pi')=( \rho_1, \ldots, \rho_p ) .\]
 This multiset determines a cuspidal $G_{k_1}\times \ldots \times G_{k_p}$-representation $\rho_1\boxtimes \ldots \boxtimes \rho_p$. Now we pick positive integers $k_{p+1}, \ldots, k_q$ such that $k_1+\ldots +k_q=n$, and pick cuspidal representations $\rho_{p+1}, \ldots, \rho_q$ of $G_{k_{p+1}}, \ldots, G_{k_q}$ respectively. Then for the inertial equivalence class 
\[\mathfrak s=[G_{k_1}\times \ldots \times G_{k_q}, \rho_1\boxtimes \ldots \boxtimes \rho_q],\]
we have that $\pi_{\mathfrak s}\neq 0$, which follows from that the bottom Bernstein-Zelevinsky layer $(  \pi^{[i^*]}\times \Pi_{i^*-1})_{\mathfrak s} \subset \pi_{\mathfrak s}$ is non-zero. Here $i^*$ is the level of $\pi$. 

Indeed for any $\mathfrak s \in \mathfrak B(G_n)$ with $\pi_{\mathfrak s} \neq 0$, $\mathfrak s$ arises in the above way. To see this, we need the following lemma:

\begin{lemma} \label{lem cusp support condition}
Let $\pi \in \mathrm{Irr}(G_{n+1})$. Let $i^*$ be the level of $\pi$. Then, for any $i\leq i^*$ with $\pi^{[i]}\neq 0$ and for any $\omega\in \mathrm{JH}(\pi^{[i]})$, $ \mathrm{csupp}(\pi^{[i^*]}) \subset \mathrm{csupp}(\omega)$ (counting multiplicities).
\end{lemma}
\begin{proof}
Let $\mathfrak m \in \mathrm{Mult}$ such that $\pi \cong \langle \mathfrak m \rangle$. Now by definition, we have that 
\[ \langle \mathfrak m \rangle \hookrightarrow \zeta(\mathfrak m) 
\]
and so $\langle \mathfrak m \rangle^{[i]} \hookrightarrow \zeta(\mathfrak m)^{[i]}$. 

For any segment $\Delta$, set $\Delta^{[0]}=\nu^{1/2}\Delta$ and set $\Delta^{[-]}=\nu^{1/2}\Delta^{-}$. We also set $\mathfrak m=\left\{ \Delta_1, \ldots , \Delta_k \right\}$ and set 
\[ \mathcal M=\left\{ \left\{ \Delta_1^{\#},\ldots, \Delta_k^{\#} \right\}: \#=[0],[-] \right\} . \]
Now the geometric lemma gives a filtration on $\zeta(\mathfrak m)^{[i]}$ whose successive quotients are isomorphic to $\zeta(\mathfrak n)$ for some $\mathfrak n  \in \mathcal M$ (see Lemma \ref{lem filt standard} below). Since $\langle \mathfrak m \rangle^{[i]} \hookrightarrow \zeta(\mathfrak m)^{[i]}$ as discussed before, any  $\omega \in \mathrm{JH}(\langle \mathfrak m \rangle^{[i]})$ is a composition factor of $\zeta(\mathfrak n)$ for some $\mathfrak n \in \mathcal M$. Hence, $\mathrm{csupp}(\omega) =\cup_{\Delta \in \mathfrak n}\Delta$ (counting multiplicities).

On the other hand, $\pi^{[i^*]}=\langle \left\{ \Delta_1^{[-]}, \ldots, \Delta_k^{[-]} \right\} \rangle$. Hence $\mathrm{csupp}(\pi^{[i^*]}) \subset \mathrm{csupp}(\omega)$ for any $\omega \in \mathrm{JH}(\pi^{[i]})$. 
\end{proof}

Now we go back to consider that $\mathfrak s \in \mathfrak B(G_n)$ with $\pi_{\mathfrak s} \neq 0$. By the Bernstein-Zelevinsky filtration, 
\[ \pi_{\mathfrak s} \neq 0 \Rightarrow (\pi^{[i]}\times \Pi_{i-1})_{\mathfrak s} \neq 0 \]
 for some $i \leq i^*$. Now, as a similar manner to what we did for the bottom layer above, we could determine (abstractly) all possible $\mathfrak s' \in \mathfrak B(G_n)$ with $( \pi^{[i]}\times \Pi_{i-1})_{\mathfrak s'}\neq 0$. Then by Lemma \ref{lem cusp support condition}, one sees that, 
\[ (\pi^{[i]} \times \Pi_{i-1})_{\mathfrak s} \neq 0 \Rightarrow ( \pi^{[i^*]}\boxtimes \Pi_{i^*-1})_{\mathfrak s} \neq 0 . \]



\subsection{Indecomposability of Zelevinsky induced modules}

Let $\mathfrak m \in \mathrm{Mult}$. In \cite{BZ}, it is shown that the restriction of $\zeta(\mathfrak m)$ to the mirabolic subgroup is strongly indecomposable. One may expect the following conjecture, which is a stronger statement of Theorem \ref{thm indecomp irred}:

\begin{conjecture}
Let $\mathfrak m$ be a multisegment for $n+1$. Then any Bernstein component of $\zeta(\mathfrak m)|_{G_n}$ is strongly indecomposable. 
\end{conjecture}

We remark that $\zeta(\mathfrak m)^{\vee}|_{G_n}$ is not strongly indecomposable in general. An example is $\mathfrak m=\left\{[\nu^{-1/2}],[\nu^{1/2}] \right\}$. In this case, the short exact sequence:
\[  0 \rightarrow \langle \Delta \rangle \rightarrow \zeta(\mathfrak m)^{\vee}|_{G_n} \rightarrow \mathrm{St}(\Delta) \rightarrow 0 ,
\]
where $\Delta=[\nu^{-1/2}, \nu^{1/2}]$. Restriction to $G_1$ gives a split sequence since $\mathrm{St}(\Delta)|_{G_1}$ is projective (see Theorem \ref{thm class restrict proj}). Hence the Iwahori component of $\zeta(\Delta)^{\vee}|_{G_1}$ is not indecomposable.

\section{Asymmetric property of left and right derivatives} \label{sec proof lemma}

We are going to prove Theorem \ref{lem same quo derivative} in this section. The idea lies in two simple cases. The first one is a generic representation. Since an irreducible generic representation is isomorphic to $\lambda(\mathfrak m) \cong \mathrm{St}(\mathfrak m)$ for $\mathfrak m \in \mathrm{Mult}$ (with the property that any two segments in $\mathfrak m$ are unlinked), a simple counting on cuspidal supports of derivatives can show Theorem \ref{lem same quo derivative} for that case. The second one is an irreducible representation whose Zelevinsky multisegment has all segments with relative length strictly greater than $1$. In such case, one can narrow down the possibility of irreducible submodule of the derivatives via the embedding $\langle \mathfrak m\rangle^{(i)} \hookrightarrow \zeta(\mathfrak m)^{(i)}$ and ${}^{(i)}\langle \mathfrak m \rangle \hookrightarrow {}^{(i)}\zeta(\mathfrak m)$, and use geometric lemma to compute the possible submodules of derivatives of $\zeta(\mathfrak m)^{(i)}$ and ${}^{(i)}\zeta(\mathfrak m)$. The combination of these two cases seems to require some extra work. The strategy is to use Speh representations, which can be viewed as a generalization of generalized Steinberg representations, and then apply Lemma \ref{lem speh multisegment 2} to obtain information on submodules.

\subsection{Union-intersection operation}

Let $\mathfrak m=\left\{ \Delta_1, \ldots , \Delta_r \right\}$. For two segments $\Delta$ and $\Delta'$ in $\mathfrak m$ which are linked, the process of replacing $\Delta$ and $\Delta'$ by $\Delta \cap \Delta'$ and $\Delta \cup \Delta'$ is called the union-intersection process. It follows from \cite[Chapter 7]{Ze} that the Zelevinsky multisegment of any irreducible composition factor in 
\[  \langle \Delta_1 \rangle \times \ldots \times \langle \Delta_r \rangle
\]
can be obtained by a chain of intersection-union process. For a positive integer $l$, define $N(\mathfrak m, l)$ to be the number of segments in $\mathfrak m$ with relative length $l$.

\begin{lemma} \label{lem segment numbers}
Let $\mathfrak m=\left\{ \Delta_1, \ldots, \Delta_r \right\} \in \mathrm{Mult}$. Let $\mathfrak m'$ be a Zelevinsky multisegment obtained from $\mathfrak m$ by a chain of union-intersection operations.  Then there exists a positive integer $l$ such that 
\[ N(\mathfrak m', l) > N(\mathfrak m, l)   ,\]
and for any $l' > l$, 
\[ N(\mathfrak m', l') \geq N(\mathfrak m, l') .\]

\end{lemma}

\begin{proof}
We shall prove inductively on the number of union-intersection operations to obtain $\mathfrak m'$ from $\mathfrak m$. We assume the targeted statement is true for $X \geq 2$ number of union-intersection operations. Let $\mathfrak m''$ be a new multisegment obtained from $\mathfrak m$ by $X+1$ number of union-intersection operations. Then we can find a multisegment $\mathfrak m'$ obtained from $\mathfrak m$ by $X$ number of union-intersection operations and $\mathfrak m''$ is obtained from $\mathfrak m'$ by one union-intersection operation. 

By inductive hypothesis, we can find a positive number $l_X$ such that 
\[ N(\mathfrak m', l_X) > N(\mathfrak m, l_X)   ,\]
and for any $l' > l_X$, 
\[ N(\mathfrak m', l') \geq N(\mathfrak m, l') .\]

Now let $\Delta_i, \Delta_j$ be the multisegments in $\mathfrak m'$ involved in the union-intersection operation to obtain $\mathfrak m''$. In particular, $\Delta_i$ and $\Delta_j$ are linked. Let $l_0$ be the relative length of $\Delta_i \cup \Delta_j$. If $l_0 \geq l_X$, set $l=l_0$, and otherwise set $l=l_X$. Now it is straightforward to check that such $l$ satisfies the required properties.

It remains to prove for $X=1$. This can be proved by a similar argument as in the last paragraph.
\end{proof}

Indeed, the proof of Lemma \ref{lem segment numbers} gives the following: 

\begin{lemma} \label{lem segment numbers 2}
Let $k \geq 2$. Let $\mathfrak m_1, \ldots \mathfrak m_k \in \mathrm{Mult}$ such that for $i=2,\ldots, k$, each $\mathfrak m_i$ is obtained from $\mathfrak m_{i-1}$ by one-step of union-intersection operation. Suppose $\Delta_i, \Delta_i'$ are two segments in $\mathfrak m_{i-1}$ involved in the union-intersection operation to obtain $\mathfrak m_i$. Let $l_i$ be the relative length of $\Delta_i \cup \Delta_i'$. Then there exists $l$ such that $l \geq l_i$ for all $i$, and 
\[  N(\mathfrak m_k, l) > N(\mathfrak m_1, l)
\]
and, for any $l' >l$,
\[  N(\mathfrak m_k, l') \geq N(\mathfrak m_1, l') .
\]
\end{lemma}

\subsection{Speh multisegments} \label{ss derivative speh}

\begin{definition}
Let $\Delta$ be a segment. Let 
\[ \mathfrak m(m,\Delta)= \left\{ \nu^{-(m-1)/2}\Delta,\nu^{1-(m-1)/2}\Delta,\ldots, \nu^{(m-1)/2}\Delta  \right\} .\]
We shall call $\mathfrak m(m,\Delta)$ to be a Speh multisegment. Define
\[  u(m,\Delta) = \langle \mathfrak m(m,\Delta) \rangle ,
\]
which will be called a Speh representation. In the literature, it is sometimes called essentially Speh representation. Denote by $L(\mathfrak m(m,\Delta))$ the relative length of $\Delta$.

We similarly define 
\[   u_r(m,i,\Delta)=\langle \nu^{-(m-1)/2}\Delta^-,\ldots, \nu^{-(m-2i+1)/2}\Delta^-, \nu^{-(m-2i-1)/2}\Delta, \ldots, \nu^{(m-1)/2}\Delta \rangle  ,\]
and
\[   u_l(m,i,\Delta)=\langle \nu^{-(m-1)/2}\Delta, \ldots,   ,\nu^{(m-2i-1)/2}\Delta, \nu^{(m-2i+1)/2}({}^-\Delta), \ldots , \nu^{(m-1)/2}({}^-\Delta) \rangle .
\]
\end{definition}

Let $l=n(\rho)$. It follows from \cite{Ta87, LM14} (also see \cite{CS17}) that 
\begin{align} \label{eqn left deriv sp}
 u(m,\Delta)^{(li)} \cong u_r(m, i, \Delta) ,
\end{align} 
and $u(m,\Delta)^{(k)}$ is zero if $l$ does not divide $k$. Applying (\ref{eqn left derivative}), we have that 
\begin{align} \label{eqn right deriv sp} {}^{(li)}u(m,\Delta) \cong u_l(m,i, \Delta) ,
\end{align}
and ${}^{(k)}u(m,\Delta) =0$ if $l$ does not divide $k$.

\subsection{A notation for multisegment}

For a multisegment $\mathfrak m=\left\{\Delta_1, \ldots, \Delta_r \right\} $, define 
\[  \mathfrak m^{(i_1, \ldots, i_r)}= \left\{ \Delta_1^{(i_1)}, \ldots, \Delta_r^{(i_r)}\right\} , \quad \mathfrak m^{(i)}= \left\{ \mathfrak m^{(i_1, \ldots, i_r)} : i_1+\ldots +i_r=i \right\},\]
\[  {}^{(i_1, \ldots, i_r)}\mathfrak m= \left\{ {}^{(i_1)}\Delta_1, \ldots, {}^{(i_r)}\Delta_r  \right \} , \quad  {}^{(i)}\mathfrak m =\left\{ {}^{(i_1, \ldots, i_r)}\mathfrak m : i_1+\ldots +i_r=i \right\} .
\]

We shall need the following lemma: 

\begin{lemma} \label{lem filt standard}
Let $\mathfrak m \in \mathrm{Mult}$. Then $\zeta(\mathfrak m)^{(i)}$ (resp. ${}^{(i)}\zeta(\mathfrak m)$) admits a filtration whose successive quotients are isomorphic to $\zeta(\mathfrak n)$ for some $\mathfrak n \in \mathfrak m^{(i)}$ (resp. $\mathfrak n \in {}^{(i)}\mathfrak m$). 
\end{lemma}
\begin{proof}
Write $\mathfrak m=\left\{ \Delta_1, \ldots, \Delta_r \right\}$. We shall assume that for $i<j$, $b(\Delta_i) \not\leq b(\Delta_j)$ and if $b(\Delta_i) \cong b(\Delta_j)$, then $a(\Delta_i) \not< a(\Delta_j)$. By definition,
\[  \zeta(\mathfrak m) \cong \langle \Delta_1 \rangle \times \ldots \times \langle \Delta_r \rangle .
\]
By geometric lemma, we have that $\zeta(\mathfrak m)^{(i)}$ admits a filtration whose successive quotients are:
\[   \langle \Delta_1 \rangle^{(i_1)} \times \ldots \times \langle \Delta_r \rangle^{(i_r)} 
\]
for $i_1+\ldots +i_r =i$. Now the lemma follows from that the product $\langle \Delta_1 \rangle^{(i_1)} \times \ldots \times \langle \Delta_r \rangle^{(i_r)}$ is isomorphic to $\zeta(\mathfrak n)$ for $\mathfrak n=\left\{ \Delta_1^{(i_1)}, \ldots , \Delta_r^{(i_r)} \right\}$.
\end{proof}


\subsection{Proof of Theorem \ref{lem same quo derivative}} \label{ss proof thm deri}

By Lemma \ref{cor generic quo sub}, it suffices to prove the statement for submodules of the derivatives. 

Let $\mathfrak m$ be the Zelevinsky multisegment with $\pi \cong \langle \mathfrak m \rangle$. We shall assume that any cuspidal representation in each segment of $\mathfrak m$ is an unramified twist of a fixed cuspidal representation $\rho$, i.e. 
\[ \mathrm{csupp}'(\langle \mathfrak m \rangle) \subset \left\{ \nu^c \rho: c\in \mathbb C \right\} . \]
 We shall prove that Theorem \ref{lem same quo derivative} for such $\pi$. The general case follows from this by writing an irreducible representation as a product of irreducible representations of such specific form. \\

\noindent
{\bf Step 1: First approximation using Lemma \ref{lem filt standard}.}
Let $\pi'$ be a common isomorphic irreducible quotient of $ \pi{}^{[i]}$ and ${}^{[i]}\pi$. (Here we assume that $ \pi{}^{[i]}$ and $ {}^{[i]}\pi$ are non-zero.) Recall that we have that 
\[   \pi \hookrightarrow \zeta(\mathfrak m)  .
\]
 Since taking derivatives is an exact functor, $\nu^{1/2} \cdot \pi^{(i)}$ embeds to $\nu^{1/2}\cdot\zeta(\mathfrak m)^{(i)}$ and so does $\pi'$.

By Lemma \ref{lem filt standard}, there is a filtration on $\zeta(\mathfrak m)^{(i)}$ given by $\zeta( \mathfrak m^{(i_1,\ldots, i_r)} )$
for $i_1+\ldots +i_r=i$, where $i_k=0$ or $n(\rho)$.  Then $\pi'$ is isomorphic to the unique submodule of $\nu^{1/2}\zeta(\mathfrak m^{(i_1',\ldots, i_r')})$ for some $(i_1',\ldots, i_r')$ and so $\pi' \cong \nu^{1/2}\langle \mathfrak m^{(i_1', \ldots, i_r')} \rangle$. Similarly, $\pi'$ is isomorphic to $\nu^{-1/2} \langle {}^{(j_1',\ldots, j_r')}\mathfrak m \rangle$ for some $j_1'+\ldots +j_r'=i$. 

 Suppose $i$ is not the level of $\pi$. Then there exists at least one $i_{k}'=0$ and at least one $j_k'=0$. Among all those segment $\Delta_k$ with either $i_k'=0$ or $j_k'=0$, we shall choose $\Delta_{k^*}$ has the largest relative (and absolute) length. Denote the relative length of $\Delta_{k^*}$ by $L$. \\

\noindent
{\bf Step 2: Second approximation using Lemma \ref{lem speh multisegment 2} }

We write $\mathfrak m$ as the sum of Speh multisegments 
\begin{equation} \label{eqn speh multiseg}
\mathfrak m= \mathfrak m_1'+\ldots + \mathfrak m_s '
\end{equation}
satisfying properties in Proposition \ref{lem speh approx}. 

 Let
\[  \mathfrak m_1 , \ldots , \mathfrak m_r \]
be all the Speh multisegments appearing in the sum (\ref{eqn speh multiseg}) such that $L(\mathfrak m_k)=L$. For each $\mathfrak m_k$, write $\mathfrak m_k=\mathfrak m(m_k, \Delta_k)$ and define $b(\mathfrak m_k)=\nu^{(m_k-1)/2}b(\Delta_k)$. We shall label $\mathfrak m_k$ in the way that $b(\mathfrak m_k) \not< b(\mathfrak m_l)$ for $k<l$. Furthermore, the labelling satisfies the properties that
\begin{enumerate}
\item[($\Diamond$)] for any $\mathfrak m_p$ and $p<q$, $\mathfrak m_p +\Delta$ is not a Speh multisegment for any $\Delta \in \mathfrak m_q$.
\item[($\Diamond \Diamond$)] if $\mathfrak m_p \cap \mathfrak m_q \neq \emptyset$ and $p \leq q$, then $\mathfrak m_q \subset \mathfrak m_p$.
\end{enumerate}

 Let $\mathfrak n_1$ be the collection of all segments in $\mathfrak m \setminus (\mathfrak m_1+\ldots +\mathfrak m_r)$ which satisfies either (1) $b(\mathfrak m_1) <b(\Delta')$ or (2) $b(\mathfrak m_1) \cong b(\Delta')$. Define inductively that $\mathfrak n_k$ is the collection of all segments $\Delta'$ in $\mathfrak m \setminus (\mathfrak m_1+\ldots +\mathfrak m_{r}+\mathfrak n_1+\ldots \mathfrak n_{k-1})$ that satisfies the property that either (1) $b(\mathfrak m_k)< b(\Delta')$, or (2) $b(\mathfrak m_k) \cong b(\Delta')$. (It is possible that some $\mathfrak n_k$ is empty.)

 By Lemma \ref{lem speh multisegment 2}, we have a series of embeddings:
\begin{align*}
  \langle \mathfrak m \rangle \hookrightarrow &  \zeta(\mathfrak n_1)\times \langle \mathfrak m_1 \rangle \times \ldots \zeta(\mathfrak n_r) \times \langle \mathfrak m_r \rangle \times \zeta( \mathfrak n_{r+1} )\\
	\hookrightarrow	&   \ldots \\
\hookrightarrow 	& \zeta(\mathfrak n_1) \times \langle \mathfrak m_1 \rangle \times \zeta(\mathfrak n_2) \times \langle \mathfrak m_2 \rangle \times \zeta(\mathfrak n_3+\ldots +\mathfrak n_r+\mathfrak m_3+\mathfrak m_r+\mathfrak n_{r+1}) \\
\hookrightarrow 	& \zeta(\mathfrak n_1) \times \langle \mathfrak m_1 \rangle \times \zeta(\mathfrak n_2+\ldots +\mathfrak n_r+\mathfrak m_2+\mathfrak m_r+\mathfrak n_{r+1}) \\
\hookrightarrow & \zeta(\mathfrak n_1 +\ldots \mathfrak n_{r+1}+\mathfrak m_1+\ldots +\mathfrak m_r)=\zeta(\mathfrak m)
\end{align*}

For simplicity, define, for $k\geq 0$,
\begin{align*}
  \lambda_k & =\zeta(\mathfrak n_1) \times \langle \mathfrak m_1 \rangle \times \ldots \times \zeta(\mathfrak n_{k})\times \langle \mathfrak m_{k} \rangle \times \zeta(\mathfrak n_{k+1}+\ldots +\mathfrak n_{r+1} +\mathfrak m_{k+1}+\ldots +\mathfrak m_{r}) 
\end{align*}

\noindent
{\bf Step 3: Approximation on right derivatives.}
Then for each $k$, we again have an embedding:
\[  \pi' \hookrightarrow \nu^{1/2}\cdot \pi^{(i)}  \hookrightarrow \nu^{1/2} \cdot \lambda_k^{(i)} .
\]
As $\lambda_k$ is a product of representations, we again have a filtration on $\lambda_k^{(i)}$. This gives that $\pi'$ embeds to a successive quotient of the filtration:
\[  \pi' \hookrightarrow \nu^{1/2}\cdot( \zeta(\mathfrak n_1)^{(p^k_1)} \times \langle \mathfrak m_1 \rangle^{(q^k_1)}  \times \ldots \times \zeta(\mathfrak n_{k})^{(p^k_{k})} \times \langle \mathfrak m_k \rangle^{(q^k_k)} \times \zeta(\mathfrak o_{k+1})^{(s^k)})
\]
with $p^k_1+\ldots +p^k_k+q_1^k+\ldots +q^k_k+s^k=i$, 
\[\mathfrak o_{k+1} = \mathfrak n_{k+1}+\ldots +\mathfrak n_{r+1}+\mathfrak m_{k+1}+\ldots +\mathfrak m_r .\] 

\begin{lemma} \label{lem exists level speh}
\begin{enumerate}
\item $i_k= 0$ for some $\Delta_k$ in $\mathfrak m$ with relative length $L$ (see the choice of $L$ in Step 1). 
\item Following above notations, there exists a $k' \geq 1$ such that at least one of $q^{k'}_l$ is not equal to the level of $\langle \mathfrak m_l \rangle$. 
\end{enumerate}
\end{lemma}

\begin{remark}
Similarly, we have $j_k=0$ for some $\Delta_k$ in $\mathfrak m$ with relative length $L$.
\end{remark}

\begin{proof}
We first assume (1) to prove (2). It suffices to show that when $k'=r$, at  least one of $q^{k'}_l$ is not equal to the level of $\langle \mathfrak m_l \rangle$. Suppose not. Let $\langle \mathfrak m_i \rangle^-$ be the highest derivative of $\langle \mathfrak m_i \rangle$. Then we obtain an embedding:
\[  \pi' \hookrightarrow \nu^{1/2}\cdot (\zeta(\mathfrak n_1)^{(p_1^k)}\times \langle \mathfrak m_1 \rangle^- \times \ldots \times \zeta(\mathfrak n_r)^{(p_r^k)}\times \langle \mathfrak m_r \rangle^- \times \zeta(\mathfrak n_{r+1})^{(s_r)}) .
\]
Now set $\mathfrak m_i^-$ to be the multisegment such that $\langle \mathfrak m_i^- \rangle \cong \langle \mathfrak m_i \rangle^-$. By definitions, $\mathfrak m_i^-$ is still a Speh multisegment. Now, by Lemma \ref{lem filt standard}, we have that
\[  \pi' \hookrightarrow \nu^{1/2}\cdot (\zeta(\widetilde{\mathfrak n}_1)\times \langle \mathfrak m_1^- \rangle \times \ldots \times \zeta(\widetilde{\mathfrak n}_r)\times \langle \mathfrak m_r^- \rangle \times \zeta(\widetilde{\mathfrak n}_{r+1})) ,
\]
where $\widetilde{\mathfrak n}_a \in \mathfrak n_a^{(p_i^k)}$. From the construction of $\mathfrak n_a$, we can check that those $\widetilde{\mathfrak n}_i$ satisfies the conditions in Lemma \ref{lem speh multisegment 2}. Hence, by Lemma \ref{lem speh multisegment 2}, 
\[  \pi' \cong \nu^{1/2} \cdot \langle \widetilde{\mathfrak n}_1+\mathfrak m_1^- +\ldots +\widetilde{\mathfrak n}_k+\mathfrak m_k^-+\widetilde{\mathfrak n}_{k+1} \rangle .
\]

Now we write $\mathfrak n_a=\left\{ \Delta_{a,1}, \ldots, \Delta_{a,r(a)} \right\}$ for each $a$. Then we have that  
\[ \widetilde{\mathfrak n}_a = \left\{ \Delta_{a,1}^{(p_{a,1})}, \ldots, \Delta_{a,r(a)}^{(p_{a,r(a)})} \right\} ,
\]
where each $p_{a,k}=0$ or $n(\rho)$. \\

\noindent
{\it Claim:} For any segment $\Delta_{a,k}$ with relative length at least $L+1$, $p_{a,k}=n(\rho)$. \\
{\it Proof of claim:} Recall that we also have $\pi' \cong \nu^{1/2} \cdot \langle \mathfrak m^{(i_1, \ldots, i_r)} \rangle$ in Step 1. Hence,
\begin{align} \label{eqn same multisegment}  \mathfrak m^{(i_1, \ldots, i_r)} = \widetilde{\mathfrak n}_1+\mathfrak m_1^- +\ldots +\widetilde{\mathfrak n}_k+\mathfrak m_k^-+\widetilde{\mathfrak n}_{k+1} .
\end{align}
Let $L^{**}$ be the largest relative length among all the relative length of segments in $\mathfrak m$. If $L^{**}=L$, then there is nothing to prove (for the claim). We assume $L^{**}>L$. In such case, there is no segment in $\mathfrak m^{(i_1, \ldots, i_r)}$ has relative length $L^{**}$. This implies that all segments $\Delta_{a,k}$ with relative length $L^{**}$ have $p_{a,k}=n(\rho)$. With the use of (\ref{eqn same multisegment}) and
\[ \mathfrak m=\mathfrak n_1 +\mathfrak m_1^- +\ldots +\mathfrak n_k+\mathfrak m_k^-+\mathfrak n_{k+1},\] we can proceed to the length $L^{**}-1$ in a similar fashion. Inductively, we obtain the claim. \\

Now we go back to the proof of the lemma. By using the claim (and the definitions of $\mathfrak m_i$), 
\[ N(\widetilde{\mathfrak n}_1 +\mathfrak m_1^- +\ldots +\widetilde{\mathfrak n}_k+\mathfrak m_k^-+\widetilde{\mathfrak n}_{k+1},L) = N(\mathfrak m, L+1). \]
 On the other hand,
\[ N(\mathfrak m^{(i_1,\ldots, i_r)}, L)\geq N(\mathfrak m, L+1)+1 \]
 by using the hypothesis in the lemma. However, $\mathfrak m=\widetilde{\mathfrak n}_1 +\mathfrak m_1^- +\ldots +\widetilde{\mathfrak n}_k+\mathfrak m_k^-+\widetilde{\mathfrak n}_{k+1}$ and this gives a contradiction. This proves (2) modulo (1). \\

It remains to prove (1). From our choice of $L$, either 
\[N({}^{(j_1, \ldots, j_r)}\mathfrak m,L) \geq  N(\mathfrak m, L+1)+1 \]
or
\[ N(\mathfrak m^{(i_1, \ldots, i_r)}, L) \geq N(\mathfrak m, L+1)+1 .
 \]
however, since $\nu^{1/2}\cdot\mathfrak m^{(i_1, \ldots, i_r)} \cong \nu^{-1/2}\cdot {}^{(j_1, \ldots, j_r)}\mathfrak m$, we have that 
\[  N(\mathfrak m^{(i_1, \ldots, i_r)}, L)=N({}^{(j_1,\ldots, j_r)}\mathfrak m, L).
\]
Thus, we must have that $N(\mathfrak m^{(i_1, \ldots, i_r)}, L) \geq N(\mathfrak m, L+1)+1$. But this forces (1).

\end{proof}

Now let $\bar{k}$ be the smallest number such that at least one of $q^{\bar{k}}_l$ is not equal to the level of $\langle \mathfrak m_l \rangle$. Now we shall denote such $l$ by $l^*$, and so $\bar{k}\geq l^*$. 

 We use similar strategy to further consider the filtrations on each $\mathfrak n_a$ by geometric lemma. For that we write, for each $a$,
\[ \mathfrak n_a =\left\{ \Delta_{a,1} ,\ldots , \Delta_{a, r(a)} \right\} 
\]
and 
\[ \mathfrak o_{\bar{k}+1}=\left\{ \Delta_{\bar{k}+1,1 }, \ldots  , \Delta_{\bar{k}+1
, r(\bar{k}+1) }\right\} . \]
Here $r(a)$ is an index counting the number of segments in $\mathfrak n_a$ depending on $a$.

Then we again have an embedding
\[  \pi' \hookrightarrow \nu^{1/2} \cdot(\zeta(\widetilde{\mathfrak n}_1) \times \langle \mathfrak m_1 \rangle^{(q_1^{\bar{k}})} \times \ldots \times \zeta(\widetilde{\mathfrak n}_{\bar{k}}) \times \langle \mathfrak m_{\bar k} \rangle^{(q_{\bar k}^{\bar{k}})} \times \zeta(\mathfrak o_{\bar k+1})^{(s_{\bar k})} ),
\]
where, for $a=1,\ldots, \bar{k}$,
\[  \widetilde{\mathfrak n}_a = \left\{ \Delta_{a,1}^{(p_{a,1})}, \ldots,  \Delta_{a, r(a)}^{(p_{a,r(a)})} \right\} \in \mathfrak n_a^{(p_{a}^{\bar{k}})}
\]
with $p_{a,1}+\ldots +p_{a,r(a)}=p_a^{\bar{k}}$ and each $p_{a,b}=0$ or $n(\rho)$, and
\[  \widetilde{\mathfrak o}_{\bar{k}+1}=\left\{ \Delta_{\bar{k}+1,1}^{(p_{\bar{k}+1,1})}, \ldots , \Delta_{\bar{k}+1,r(\bar{k}+1)}^{(p_{\bar k+1,r(\bar{k}+1)})} \right\} \in \mathfrak o_{\bar{k}+1}^{(s_{\bar{k}})}
\] 
with $p_{\bar{k}+1,1}+\ldots+p_{\bar{k}+1, r(\bar{k}+1)}=s_{\bar{k}}$, with each $p_{\bar{k}+1,b}=0$ or $n(\rho)$. \\

\noindent
{\bf Step 4: Computing some indexes $p_{a,b}$ and approximating the number of special segments by union-intersection operations.}
We claim (*) that if $\Delta_{a,b}$ has a relative length at least $L+1$, then $p_{a,b}=n(\rho)$. This is indeed similar to the proof of Lemma \ref{lem exists level speh} and the main difference is that we do not have an analogous form of (\ref{eqn same multisegment}) (since we cannot apply Lemma \ref{lem speh multisegment 2}). Instead, we can obtain this from Lemma \ref{lem segment numbers 2} (and its proof of Lemma \ref{lem segment numbers}).



Now from our choice of $\bar{k}$, we have that $\langle \mathfrak m_{l^*} \rangle^{(q^{\bar k}_{l^*})}$ is not a Speh representation. We can write 
\[ \mathfrak m_{l^*} =\left\{ \nu^{-x+1}\Delta^*, \ldots , \Delta^* \right\} \]
for a certain $\Delta^*$ with relative length $L$ and some $x$. By ($\Diamond$), $\nu \Delta^* \notin \mathfrak m_l$ for any $l \geq l^*$ from our labelling on $\mathfrak m_l$. Rephrasing the statement, we get the following statement: \\
\begin{center}
(**) {\it  $ \nu^{1/2}\Delta^* \notin \nu^{-1/2}\mathfrak m_l$ for any $l \geq l^*$ . }
\end{center}
Now with (*), we have that $\pi'$ is a composition factor of $\zeta(\mathfrak m''')$, where $\mathfrak m'''$ is some multisegment containing all the segments $\Delta^-$ (counting multiplicities) with $\Delta$ in $\mathfrak m$ that has relative length at least $L+1$ and containing an additional segment $\nu^{1/2}\Delta^*$, and we shall call the former segments (i.e. the segment in the form of $\Delta^-$) to be {\it special} for convenience.

We can apply the intersection-union process to obtain the Zelevinsky multisegment for $\pi'$ from $\mathfrak m'''$. However in each step of the process, any one of the two segments involved in the intersection-union cannot be special. Otherwise, by Lemma \ref{lem segment numbers 2}, there exists $l \geq L+1$ such that the number of segments in the resulting multisegment with relative length $l$ is more than the number of segments in $\mathfrak m^{(i_1,\ldots, i_r)}$. (Here we also used that $L(\mathfrak m^{(i_1,\ldots, i_r)}, l)=L(\mathfrak m''',l)$ for any $l\geq L+1$.) Hence we obtain the following:
\begin{center}
 (***) {\it  the number of segments $\Delta$ in $\nu^{1/2}\cdot \mathfrak m^{(i_1,\ldots, i_r)}$ such that $\nu^{1/2}\Delta^* \subset \Delta$ is at least equal to {\bf one plus} the number special segments in $\mathfrak m'''$ satisfying the same property } \\
\end{center}

\ \\

\noindent
{\bf Step 5: Approximation on left derivatives.} Now we come to the final part of the proof. We now consider $\nu^{-1/2}\cdot {}^{(i)}\pi$. Following the strategy for right derivatives, we have that for each $k$
\[  \pi' \hookrightarrow \nu^{-1/2}\cdot{}^{(i)}\pi \hookrightarrow \nu^{-1/2}\cdot {}^{(i)}\lambda_k .
\]
This gives that $\pi'$ 
\[  \pi' \hookrightarrow \nu^{-1/2} \cdot ({}^{(u^k_1)}\zeta(\mathfrak n_1) \times {}^{(v^k_1)}  \langle \mathfrak m_1 \rangle  \times \ldots \times {}^{(u^k_{k})}\zeta(\mathfrak n_{k}) \times  {}^{(v^k_{k})}\zeta(\mathfrak m_{k})\times {}^{(w^{k})}\zeta(\mathfrak o_{k+1}))
\]
with $u^k_1+\ldots +u^k_k+v^k_1+\ldots +v^k_k+w^k=i$. Let $\widetilde{k}$ be the smallest integer such that at least one of $v^{\widetilde{k}}_l$ is not equal to the level of $\langle \mathfrak m_l \rangle$. (For the existence of $\widetilde{k}$, see Lemma \ref{lem exists level speh}.) We shall denote such $l$ by $\widetilde{l}$. \\

\noindent
{\bf Step 6: the case that $\widetilde{k} \geq \bar{k}$ and determining the number of segments.}


We firstly consider the case that $\widetilde{k} \geq \bar{k}$. Similar to left derivatives, we have that
\[  \pi' \hookrightarrow \nu^{-1/2} \cdot \zeta(\widehat{\mathfrak n}_1) \times {}^{(v_1)}\langle \mathfrak m_1 \rangle \times \ldots \times \zeta(\widehat{\mathfrak n}_{\bar{k}-1}) \times {}^{(v_{\bar k-1})}\langle \mathfrak m_{\bar k-1} \rangle \times {}^{(w_{\bar k-1})} \zeta(\widehat{\mathfrak o}_{\bar k}),
\]
where 
\[  \widehat{\mathfrak n}_a = \left\{ {}^{(u_{a,1})}\Delta_{a,1}, \ldots,  {}^{(u_{a,r(a)})}\Delta_{a, r(a)} \right\} .
\]
with $u_{a,1}+\ldots +u_{a,r(a)}=u_a$ and each $u_{a,b}=0$ or $n(\rho)$. Since we assume that $\widetilde{k}\geq \bar{k}$, we have that ${}^{(v_l)}\langle \mathfrak m_l \rangle$ is a highest derivative for any $l\leq \bar{k}-1$ and so is a Speh representation, and we can apply Lemma \ref{lem speh multisegment 2}(1). Hence the unique subrepresentation of 
\[\nu^{-1/2}\cdot \zeta(\widehat{\mathfrak n}_1) \times {}^{(v_1)}\langle \mathfrak m_1 \rangle \times \ldots \times \zeta(\widehat{\mathfrak n}_{\bar{k}-1}) \times {}^{(v_{\bar k-1})}\langle \mathfrak m_{\bar k-1} \rangle  \times \zeta(\widehat{\mathfrak o}_{\bar{k}})\]
is isomorphic to
\begin{equation}\label{eqn isomorphism} \nu^{-1/2} \cdot \langle \widehat{\mathfrak n}_1+\widehat{\mathfrak m}_1+\ldots + \widetilde{\mathfrak n}_{\bar k-1}+\widehat{\mathfrak m}_{\bar k-1}+\widehat{\mathfrak o}_{\bar k} \rangle ,
\end{equation}
where $\langle \widehat{\mathfrak m}_l \rangle ={}^{(v_l)}\langle \mathfrak m_l \rangle$. Similar to (*) for right derivatives (but the proof could be easier here), we obtain the analogous statement for those $\widehat{\mathfrak n}_a$.  Now if 
\[\Delta \in \nu^{-1/2}( \widehat{\mathfrak n}_1+\widehat{\mathfrak m}_1+\ldots + \widetilde{\mathfrak n}_{\bar k-1}+\widehat{\mathfrak m}_{\bar k-1}+\widehat{\mathfrak o}_k )\]
such that $\Delta=\nu^{1/2}\Delta^*$, then we must have that $\Delta =\nu^{-1/2}\cdot {}^-\Delta_0$ or $\nu^{-1/2}\cdot \Delta_0$ for some segment $\Delta_0$ in $\mathfrak m$. For the latter case to happen, we must have that $\Delta_0 \in \mathfrak m_l \subset \mathfrak o_{\bar{k}}$ with $l \geq \bar{k}\geq l^* $, but, by (**), the possibility $\Delta=\nu^{-1/2}\Delta_0$ cannot happen. Thus we must have that $\Delta$ is a special in the same sense as the discussion in right derivatives. This concludes the following:
\begin{center}
(****) {\it The number of segments $\Delta$ in $\nu^{-1/2}( \widehat{\mathfrak n}_1+\widehat{\mathfrak m}_1+\ldots + \widetilde{\mathfrak n}_{\bar k-1}+\widehat{\mathfrak m}_{\bar k-1}+\widehat{\mathfrak o}_{\bar k})$ with the property that $\nu^{1/2}\Delta^* \subset \Delta$ is equal to the number of special segments satisfying the same property. }
\end{center}
Now the above statement contradicts to (***) since both Zelevinsky multisegments give an irreducible representation isomorphic to $\pi'$. \\

\begin{remark}
The isomorphism (\ref{eqn isomorphism}) is a key to obtain the equality in (****), in comparison with the inequality in (***).
\end{remark}

\noindent
{\bf Step 7: the case that $\bar{k} >\widetilde{k}$.} Now the way to get contradiction in the case $\bar{k} \geq \widetilde{k}$ is similar by interchanging the role of left and right derivatives. We remark that to prove the analogue of (**), one uses ($\diamond$). And  to obtain the similar isomorphism as (\ref{eqn isomorphism}), one needs to use Lemma \ref{lem speh multisegment 2}(2). We can argue similarly to get an analogue of (***) and (****). Hence the only possibility that $\nu^{1/2}\cdot \pi^{(i)}$ and so $\nu^{-1/2}\cdot {}^{(i)}\pi$ have an isomorphic irreducible quotient only if $i$ is the level for $\pi$. This completes the proof of Theorem \ref{ss proof thm deri}.

\subsection{Branching law in opposite direction}

Here is another consequence of the asymmetric property on the Hom-branching law in another direction:

\begin{corollary} \label{cor branching law direction}
Let $\pi' \in \mathrm{Irr}(G_n)$. Let $\pi \in \mathrm{Irr}(G_{n+1})$.
\begin{enumerate}
\item Suppose $\pi$ is not $1$-dimensional. Then 
\[ \mathrm{Hom}_{G_n}(\pi', \pi|_{G_n}) = 0 .
\]
\item Suppose $\pi$ is $1$-dimensional. Then 
\[ \mathrm{Hom}_{G_n}(\pi', \pi|_{G_n}) \neq 0
\]
if and only if $\pi'$ is also $1$-dimensional and $\pi'=\pi|_{G_n}$.
\end{enumerate}
\end{corollary}

\begin{proof}
The second statement is trivial. We consider the first one. Since $\pi$ is not one-dimensional, the level of $\pi$ is not $1$ by Zelevinsky classification. By Theorem \ref{lem same quo derivative} and Proposition \ref{prop multi free}, $ \pi^{[1]}$ and $ {}^{[1]}\pi$ have no common irreducible submodule if $\pi^{(1)} \neq 0$ and ${}^{(1)}\pi \neq 0$. Then at least one of
\[  \mathrm{Hom}_{G_n}(\pi', \pi^{[1]})=0 \quad \mbox{or} \quad \mathrm{Hom}_{G_n}(\pi',  {}^{[1]}\pi)=0 .
\]
On the other hand, we have that for all $i \geq 2$, by Frobenius reciprocity,
\[  \mathrm{Hom}_{G_n}(\pi', \mathrm{ind}_{R_{n-i+1}}^{G_n} \pi^{[i]} \boxtimes \psi_{i-1}) =\mathrm{Hom}_{G_{n+1-i}\times G_{i-1}}( \pi'_{N_{i-1}},  \pi^{[i]} \boxtimes \Pi_{i-1}) ,
\]
and
\[  \mathrm{Hom}_{G_n}(\pi', \mathrm{ind}_{R_{n-i+1}^-}^{G_n}  {}^{[i]}\pi \boxtimes \psi_{i-1}) =\mathrm{Hom}_{G_{i-1} \times G_{n+1-i} }(\pi'_{N_{n+1-i}}, \Pi_{i-1}\boxtimes  {}^{[i]}\pi) 
\]
However, in the last two equalities, the Hom in the right hand side is always zero (see Lemma \ref{lem zero hom gg}). Now applying a Bernstein-Zelevinsky filtration implies the corollary.


\end{proof}

\section{Preserving indecomposability of Bernstein-Zelevinsky induction} \label{s preserve indecomp}

We shall use  the following criteria of indecomposable representations:

\begin{lemma} \label{lem algebra indcompose crit}
Let $G$ be a connected reductive $p$-adic group. Let $\pi$ be a smooth representation of $G$. The only idempotents in $\mathrm{End}_G(\pi)$ are $0$ and the identity (up to automorphism) if and only if $\pi$ is indecomposable.
\end{lemma}

\begin{proof}
If $\pi$ is not indecomposable, then any projection to a direct summand gives a non-identity idempotent. On the other hand, if $\sigma \in \mathrm{End}_G(\pi)$ is a non-identity idempotent (i.e. $\sigma(\pi) \neq \pi$), then $\pi \cong \mathrm{im}(\sigma) \oplus \mathrm{im}(1-\sigma)$.
\end{proof}

\begin{lemma} \label{lem isomorphic to center}
Set $\Pi=\Pi_k$. Let $\mathfrak s \in \mathfrak B(G_k)$. Then ${}^{(k)}(\Pi_{\mathfrak s}) \cong \mathfrak Z_{\mathfrak s}$, where $\mathfrak Z_{\mathfrak s}$ is the Bernstein center of $\mathfrak R_{\mathfrak s}(G_k)$.
\end{lemma}

\begin{proof}
We have 
\[ {}^{(k)}(\Pi_{\mathfrak s}) \cong \mathrm{Hom}_{\mathbb C}(\mathbb C, {}^{(k)}(\Pi_{\mathfrak s})) \cong \mathrm{Hom}_{G_k}(\Pi, \Pi_{\mathfrak s}) \cong \mathrm{End}_{G_k}(\Pi_{\mathfrak s}) ,
\]
where the second isomorphism follows from the second adjointness theorem of \cite{CS18}, and the third isomorphism follows from the Bernstein decomposition. Now the lemma follows from that the final endomorphism is isomorphic to $\mathfrak Z_{\mathfrak s}$ by \cite{BH}.
\end{proof}

\begin{theorem} \label{thm indecomposable bz}
Let $\pi$ be an admissible indecomposable smooth representation of $G_{n-i}$. For each $\mathfrak{s} \in \mathfrak{B}(G_n)$, the Bernstein component $(\mathrm{ind}_{R_{n-i}}^{G_n} \pi \boxtimes \psi_i)_{\mathfrak{s}}$ is indecomposable. 
\end{theorem}

\begin{proof}

We first prove the following:
\begin{lemma} \label{lem induction change}
Let $\pi \in \mathrm{Alg}(G_{n-i})$ be admissible. Then, as endomorphism algebras,
\[  \mathrm{End}_{G_n}(\mathrm{ind}_{R_{n-i}}^{G_n} \pi \boxtimes \psi_i ) \cong \mathrm{End}_{G_{n-i}}(\pi ) \otimes \mathfrak{Z}_i ,\]
where $\mathfrak Z_i=\prod_{\mathfrak s \in \mathfrak B(G_{i})} \mathfrak Z_{\mathfrak s}$ is the Bernstein center.
\end{lemma}

\begin{proof}

 Let $\mathfrak s_1, \mathfrak s_2 \in \mathfrak B(G_{n-i})$. Let $\pi_1, \pi_2$ be two admissible $G_{n-i}$ representations in $\mathfrak R_{\mathfrak s_1}(G_{n-i})$ and $\mathfrak R_{\mathfrak s_2}(G_{n-i})$ respectively. Then
\begin{align*}
    &\mathrm{Hom}_{G_n}(\mathrm{ind}_{R_{n-i}}^{G_n}\pi_1\boxtimes \psi_i, \mathrm{ind}_{R_{n-i}}^{G_n}\pi_2 \boxtimes \psi_i) \\
 \cong & \prod_{\mathfrak s \in \mathfrak R(G_n)}\mathrm{Hom}_{G_n}(\mathrm{ind}_{R_{n-i}}^{G_n} \pi_1 \boxtimes \psi_i, (\mathrm{ind}_{R_{n-i}}^{G_n}\pi_2 \boxtimes \psi_i)_{\mathfrak s}) \\
 \cong &  \prod_{\mathfrak s \in \mathfrak R(G_n)}\mathrm{Hom}_{G_n}(\pi_1, {}^{(i)}(\mathrm{ind}_{R_{n-i}}^{G_n}\pi_2 \boxtimes \psi_i)_{\mathfrak s}) \quad \mbox{\cite[Lemma 2.2]{CS18} }\\
 \cong  &  \prod_{\mathfrak u \in \mathfrak R(G_i)}\mathrm{Hom}_{G_{n-i}}(\pi_1, {}^{(i)}(\pi_2 \times (\Pi_i)_{\mathfrak u})) 
\end{align*}
We remark that in order to apply the second adjointness theorem \mbox{\cite[Lemma 2.2]{CS18} }, we need to get the first isomorphism. For the third isomorphisms, see discussions in Section \ref{ss inertial equ class}.

Now we apply the geometric lemma on ${}^{(i)}(\pi_2 \times (\Pi_i)_{\mathfrak u})$ and we obtain a filtration  on ${}^{(i)}(\pi_2 \times (\Pi_i)_{\mathfrak u})$ whose successive quotients are of the form 
\[ {}^{(j)}(\pi_2) \times {}^{(k)}((\Pi_i)_{\mathfrak u}) ,
\]
where $j+k=i$. Using Lemma \ref{lem zero hom gg}, 
\[  \mathrm{Hom}_{G_{n-i}}(\pi_1, {}^{(j)}(\pi_2) \times {}^{(k)}((\Pi_i)_{\mathfrak u})) =0
\]
unless $j=0$ and $k=i$. Combining the above isomorphisms, we now have
\begin{align*}
    &\mathrm{Hom}_{G_n}(\mathrm{ind}_{R_{n-i}}^{G_n}\pi_1\boxtimes \psi_i, \mathrm{ind}_{R_{n-i}}^{G_n}\pi_2 \boxtimes \psi_i) \\
\cong & \prod_{\mathfrak u \in \mathfrak B(G_i)} \mathrm{Hom}_{G_{n-i}}(\pi_1, \pi_2 \otimes {}^{(i)}((\Pi_i)_{\mathfrak u})) \\
\cong & \prod_{\mathfrak u \in \mathfrak B(G_i)} \mathrm{Hom}_{G_{n-i}}(\pi_1, \pi_2\otimes \mathfrak Z_{\mathfrak u}) \\
\cong & \prod_{\mathfrak u \in \mathfrak B(G_i)} \mathrm{Hom}_{G_{n-i}}(\pi_1,\pi_2) \otimes \mathfrak Z_{\mathfrak u} \\
\cong & \mathrm{Hom}_{G_{n-i}}(\pi_1,\pi_2) \otimes \mathfrak Z
\end{align*}

We remark that the third and last isomorphisms require the admissibility of $\pi_1$ and $\pi_2$. The second isomorphism follows from Lemma \ref{lem isomorphic to center}. One further traces the above isomorphisms to see that they give an isomorphism preserving the algebra structure and we omit the details. Now we specialize to $\pi_1=\pi_2=\pi$ and we obtain the lemma when each of $\pi_1, \pi_2$ lying in one Bernstein component.

The general case follows by writing $\pi=\oplus_{\mathfrak s} \pi_{\mathfrak s}$.
\end{proof}

 Now the theorem follows from Lemmas \ref{lem algebra indcompose crit} and \ref{thm indecomposable bz} since the right hand side has only 0 and 1 as idempotents if and only if the left hand side does.
\end{proof}

Let $M_n$ be the mirabolic subgroup in $G_{n+1}$, that is the subgroup containing all matrices with the last row of the form $(0,\ldots, 0,1)$. We have the following consequence restricted from $M_n$ to $G_n$. Here $G_n$ is viewed as the subgroup of $M_n$ via the embedding $g \mapsto \mathrm{diag}(g,1)$. 

\begin{corollary}
Let $\pi$ be an irreducible smooth representation of $M_n$. Then for any $\mathfrak s \in \mathfrak B(G_n)$, $\pi_{\mathfrak s}$ is indecomposable.
\end{corollary}

\begin{proof}
This follows from \cite[Corollary 3.5]{BZ} and Theorem \ref{thm indecomposable bz}.
\end{proof}

\section{Appendix: Speh representation approximation} \label{ss app speh approx}

\subsection{Two lemmas}

Some studies on parabolic induction can be found in \cite{Ta15, LM16}. We give a proof of the following specific cases, using the theory of derivatives.  We use notations in Section \ref{ss derivative speh}.

\begin{lemma} \label{lem switch Speh}
Let $m\in \mathbb{Z}$. Let $\Delta=[\nu^a \rho, \nu^b \rho]$. If $\Delta' = [\nu^k \rho, \nu^{l}\rho]$ such that
\[ -\frac{m-1}{2}+a \leq k \leq l \leq \frac{m-1}{2}+b , \] then 
\[ u(m,\Delta) \times \langle \Delta' \rangle \cong  \langle \Delta' \rangle \times u(m,\Delta) ,\]
and is  irreducible.
\end{lemma}

\begin{proof}
When $-(m-1)/2+a-k \notin \mathbb{Z}$, this case is easier and we omit the details. We now consider $-(m-1)/2+a-k \in \mathbb{Z}$. 

The statement is not difficult when $\Delta$ is singleton $[\rho]$ because 
\[ u(m,\Delta) \cong \mathrm{St}([\nu^{-(m-1)/2}\rho, \nu^{(m-1)/2}\rho]) \]
 and $ \mathrm{St}(\Delta') \times \langle \Delta'' \rangle \cong \langle \Delta'' \rangle \times \mathrm{St}(\Delta')$ whenever $\Delta'' \subset \Delta'$. (Indeed, one can also prove the latter fact by similar arguments as below by noting that the Zelevinsky multisegment of any simple composition factor in $\mathrm{St}(\Delta')\times \langle \Delta''\rangle$ contains a segment $\widetilde{\Delta}$ with $b(\widetilde{\Delta}) \cong b(\Delta')$ and at least one segment $\widehat{\Delta}$ with $b(\widehat{\Delta})\cong b(\Delta'')$.)

We now assume $\Delta$ is not a singleton. We consider two cases:
\begin{enumerate}
\item[(1)] Case 1: $l=b+\frac{m-1}{2}$. Suppose $\tau$ is a composition factor of $a(m,\Delta ) \times \langle \Delta' \rangle$ with the associated Zelevinsky multisegment $\mathfrak m$. Then we know that at least two segments $\Delta_1, \Delta_2$ in $\mathfrak m$ takes the form $b(\Delta_1)\cong b(\Delta_2) \cong \nu^{\frac{m-1}{2}+b}\rho$. If $\tau^{(i^*)}$ is the highest derivative of $\tau$, then we know that the cuspidal support of $\tau^{(i^*)}$ does not contain $\nu^{\frac{m-1}{2}}\rho$. We also have that $\tau^{(i^*)}$ is a composition factor of $(u(m,\Delta) \times \langle \Delta \rangle)^{(i^*)}$. The only possibility is that $i^*=m+1$  i.e  $(u(m,\Delta )\times \langle \Delta' \rangle)^{(i^*)}=u(m,\Delta^-) \times \langle (\Delta')^- \rangle$, which the latter one is irreducible by induction. This proves the lemma. Since taking derivative is an exact functor, we have that $u(m,\Delta ) \times \langle \Delta' \rangle$ is irreducible. Using the Gelfand-Kazhdan involution \cite[Section 7]{BZ1}, we have that $\langle \Delta' \rangle \times u(m,\Delta) \cong u(m,\Delta) \times \langle \Delta' \rangle$.
\item[(2)] Case 2: $l<b+\frac{m-1}{2}$. The argument is similar. Again suppose $\tau$ is a composition factor of $u(m,\Delta) \times \langle \Delta' \rangle$. Using an argument similar to above, we have that the level of $\tau$ is either $m+1$ or $m$. However, if the level of $\tau$ is $m$, then $\tau^{(m)}$ would be $u(m,\Delta^-) \times \langle \Delta' \rangle$, which is irreducible by induction. Then it would imply that the number of segments in Zelevinsky multisegment of $\tau^{(m)}$ is $m$ and contradicts that the number of segments for the Zelevinsky multisegment of the highest derivative of an irreducible representation $\pi$ must be at most that for $\pi$. Hence, the level of $\tau$ must be $m+1$. Now repeating a similar argument as in (1) and using the induction, we obtain the statements. 
\end{enumerate}

\end{proof}

\begin{lemma} \label{lem speh commute with one more}
Let $\Delta$ be a segment. Then 
\[   u(m,\Delta^-) \times \nu^{(m-1)/2}\langle \Delta \rangle \cong \nu^{(m-1)/2}\langle \Delta \rangle \times u(m, \Delta^-)
 \]
is irreducible.
\end{lemma}

\begin{proof}

The statement is clear if $\Delta$ is a singleton. For the general case, we note by simple counting that the Zelevinsky multisegment of any composition factor must contain a segment $\Delta_1$ with $b(\Delta_1) \cong \nu^{(m-1)/2}b(\Delta)$ and at least one segment $\Delta_2$ with $b(\Delta_2)\cong \nu^{(m-1)/2}b(\Delta^-)$. Then one proves the statement by a similar argument using highest derivative as in the previous lemma.

\end{proof}

\subsection{Speh representation approximation}

For a Speh multisegment 
\[ \mathfrak m=\left\{ \Delta, \nu^{-1}\Delta, \ldots, \nu^{-k}\Delta \right\}, \]
define $b(\mathfrak m)=b(\Delta)$.

\begin{proposition}  \label{lem speh approx}
 Let $\mathfrak{m}=\left\{ \Delta_1,\ldots, \Delta_k \right\}$. Then there exists Speh multisegments $\mathfrak{m}_1, \ldots \mathfrak{m}_r$ satisfying the following properties:
\begin{enumerate}
\item $\mathfrak{m}=\mathfrak{m}_1+\ldots +\mathfrak{m}_r$;
\item For each Speh multisegment $\mathfrak{m}_i$ and any $j >i$, there is no segment $\Delta$ in $\mathfrak{m}_i$ such that $\mathfrak{m}_j +\left\{ \Delta \right\}$ is a Speh multisegment;
\item $\langle \mathfrak{m} \rangle$ is the unique submodule of $\langle \mathfrak{m}_1 \rangle \times \ldots \times \langle \mathfrak{m}_r \rangle$;
\item $b(\mathfrak m_i) \not\leq b(\mathfrak m_j)$ if $i<j$
\item if $\mathfrak m_i \cap \mathfrak m_j \neq \emptyset$ and $i\leq j$, then $\mathfrak m_j \subset \mathfrak m_i$.
\end{enumerate}
\end{proposition}

\begin{proof}
We shall label the segments $\mathfrak{m}$ in the way that for $i <j$, either (i) $b(\Delta_i) \not< b(\Delta_j)$ and (ii) if $b(\Delta_i)=b(\Delta_j)$, then $a(\Delta_i) \not< a(\Delta_j)$. Let $\Delta=\Delta_1$. Let $k$ be the largest integer ($k \geq 0$) such that $\Delta, \nu^{-1}\Delta, \ldots, \nu^{-k}\Delta$ are segments in $\mathfrak{m}$. We claim that 
\begin{equation} \label{eqn embedding speh}
  \langle \mathfrak{m} \rangle \hookrightarrow    \langle \mathfrak{m}' \rangle \times \langle \mathfrak{m} \setminus \mathfrak{m}'  \rangle 
\end{equation}
and moreover $\langle \mathfrak{m} \rangle$ is the unique submodule of $\langle \mathfrak{m}' \rangle \times \langle \mathfrak{m} \setminus \mathfrak{m}' \rangle$. By induction, the claim proves (1) and (3), and the remaining ones follow from the inductive construction.

We now prove the claim. We shall prove by induction that for $i=0,\ldots, k$, 
\[ \langle \mathfrak{m} \rangle \hookrightarrow \langle \mathfrak{m}_i \rangle \times \zeta(\mathfrak{m} \setminus \mathfrak{m}_i ) \hookrightarrow \zeta(\mathfrak m), \]
where $\mathfrak{m}_i = \left\{ \Delta, \nu^{-1} \Delta, \ldots, \nu^{-i}\Delta \right\}$. When $i=0$ the statement is clear from the definition. Now suppose that we have the inductive statement for $i$. To prove the statement for $i+1$. We now set $\mathfrak{n}$ to be the collection of all segments $\Delta'$ in $\mathfrak{m}$ such that $b(\Delta')=b(\Delta), \nu^{-1}b(\Delta), \ldots, \nu^{-i}b(\Delta)$ and $\nu^{-i}a(\Delta)$ proceeds $a(\Delta')$. We also set $\bar{\mathfrak{n}}$ to be the collection of all segments $\Delta'$ in $\mathfrak{m}$ such that $b(\Delta')=b(\Delta), \nu^{-1}b(\Delta),\ldots, \nu^{-i}b(\Delta)$ and $a(\Delta')$ precedes $\nu^{-i-1}a(\Delta)$. By using (\ref{eqn commute product trivial} ) several times, we have that $\zeta(\mathfrak{n})\times \zeta(\bar{\mathfrak{n}}) \cong \zeta(\bar{\mathfrak{n}}+\mathfrak{n})$ and 
\begin{align} \label{eqn first separate iso}    \zeta(\mathfrak{m} \setminus \mathfrak{m}_i)  \cong \zeta(\mathfrak{n}) \times \zeta(\bar{\mathfrak{n}}) \times \zeta(\mathfrak{m}\setminus (\mathfrak{n}+\bar{\mathfrak{n}}+\mathfrak{m}_i)).\end{align}
From our construction (and $i \neq k$), we have that $\nu^{-i-1}\Delta$ is a segment in $\mathfrak{m}\setminus (\mathfrak{n}+\bar{\mathfrak{n}}+\mathfrak{m}_i)$ and 
\[ \zeta(\mathfrak{m}\setminus (\mathfrak{n}+\bar{\mathfrak{n}}+\mathfrak{m}_i)) =\langle \nu^{-i-1}\Delta \rangle \times \zeta(\mathfrak{m}\setminus (\mathfrak{n}+\bar{\mathfrak{n}}+\mathfrak{m}_{i+1}) )\]
On the other hand
\begin{align} 
& \langle \mathfrak{m}_i \rangle \times \zeta(\mathfrak{n}) \times \zeta(\bar{\mathfrak{n}}) \times \langle \nu^{-i-1}\Delta \rangle \\
\cong &  \zeta(\mathfrak{n}) \times \langle \mathfrak{m}_i \rangle \times \zeta(\bar{\mathfrak{n}}) \times \langle \nu^{-i-1}\Delta \rangle \\
\cong &  \zeta(\mathfrak{n}) \times \langle \mathfrak{m}_i \rangle \times \langle  \nu^{-i-1}\Delta \rangle \times \zeta(\bar{\mathfrak{n}}) \\
\hookleftarrow &  \zeta(\mathfrak{n}) \times \langle \mathfrak{m}_{i+1} \rangle \times \zeta(\bar{\mathfrak{n}})\\
\cong &  \langle \mathfrak{m}_{i+1} \rangle \times \zeta(\mathfrak{n}) \times \zeta(\bar{\mathfrak{n}})
\end{align}

The first and last isomorphism is by Lemma \ref{lem speh approx}. The injectivity in forth line comes from the uniqueness of submodule in $\zeta(\mathfrak{m}_{i+1})$. The second isomorphism follows from again by (\ref{eqn commute product trivial}). Combining (\ref{eqn first separate iso}), the above series of isomorphisms and the inductive case, we have:
\[ \zeta(\mathfrak m) \hookleftarrow \langle \mathfrak{m}_{i+1} \rangle \times \zeta(\mathfrak{n}) \times \zeta(\bar{\mathfrak{n}}) \times \zeta(\mathfrak{m}\setminus (\mathfrak{n}+\bar{\mathfrak{n}}+\mathfrak{m}_{i+1})).
\]
This gives the desired injectivity by using the uniqueness of submodule of $\zeta(\mathfrak m)$ and proves the claim.

\end{proof}

We shall need a variation which is more flexible in our application.

\begin{lemma} \label{lem speh multisegment 2}
Let $\mathfrak m=\mathfrak m(m,\Delta)$ be a Speh multisegment. Let $\mathfrak n_1$ be a Zelevinsky multisegment such that for any segment $\Delta'$ in $\mathfrak n_1$ satisfying $b(\Delta) < b(\Delta')$ or $b(\Delta) \cong b(\Delta')$. Let $\mathfrak n_2$ be a Zelevinsky multisegment such that for any segment $\Delta'$ in $\mathfrak n_2$ satisfying either one of the following properties:
\begin{enumerate}
\item  $b(\Delta) \not<b(\Delta')$; or
\item  $b(\Delta)\not< b(\Delta')$, or if $b(\Delta)< b(\Delta')$, then $(\Delta')^-=\Delta$.
\end{enumerate}
Then
\[\langle \mathfrak n_1+\mathfrak m+\mathfrak n_2 \rangle \hookrightarrow  \zeta(\mathfrak n_1) \times \langle \mathfrak m \rangle \times \zeta(\mathfrak n_2) \hookrightarrow \zeta(\mathfrak n_1+\mathfrak m+\mathfrak n_2) .
\]
In particular, $\zeta(\mathfrak n_1) \times \langle \mathfrak m \rangle \times \zeta(\mathfrak n_2)$ has unique submodule isomorphic to $\langle \mathfrak n_1 +\mathfrak m+\mathfrak n_2 \rangle$.
\end{lemma}

\begin{remark}
Case (2) covers case (1). But for the purpose of clarity of an argument used in Section \ref{ss proof thm deri}, we divide into two cases. We also recall that the case $b(\Delta) \not\leq b(\Delta')$ includes the possibility $b(\Delta) \cong b(\Delta')$.
\end{remark}

\begin{proof}
For all cases, we have that
\[   \zeta(\mathfrak n_1) \times \zeta(\mathfrak m+\mathfrak n_2) \cong \zeta(\mathfrak n_1+\mathfrak m+\mathfrak n_2) .
\]
Using (\ref{eqn embedding speh}) for (1) we obtain the lemma. For (2), let $\mathfrak n'$ be all the segments in $\mathfrak n_2$ with the property that $(\Delta')^-\cong \Delta$. Then we have that
\[ \zeta(\mathfrak n') \times \zeta(\mathfrak m)\times \zeta(\mathfrak n_2 \setminus \mathfrak n') \hookrightarrow \zeta(\mathfrak n') \times \zeta((\mathfrak m+\mathfrak n_2)\setminus \mathfrak n')\hookrightarrow  \zeta(\mathfrak m+\mathfrak n_2) .
\]
By Lemma \ref{lem speh commute with one more}, we have that 
\[ \zeta(\mathfrak n') \times \zeta(\mathfrak m)\times \zeta(\mathfrak n_2 \setminus \mathfrak n')  \cong \zeta(\mathfrak m)\times \zeta(\mathfrak n')\times \zeta(\mathfrak n_2 \setminus \mathfrak n') \cong \zeta(\mathfrak m)\times \zeta(\mathfrak n_2),\]
which proves the lemma.
\end{proof}



\section{Appendix}

 Let $G$ be a connected reductive group over a non-Archimedean local field. 

\begin{lemma} \label{lem subquot}
Let $\pi$ be a smooth representation of $G$. Let $\pi$ admits a filtration  
\[   0=\pi_0' \subset \pi_1'  \subset \ldots \subset \pi_r' =\pi 
\] Suppose $\pi$ admits an irreducible subquotient $\tau$. Let $\pi_k=\pi_k'/\pi_{k-1}'$ for $k=1,\ldots, r$. Then there exists $s$ such that $\pi_s$ contains an irreducible subquotient isomorphic to $\tau$.

\end{lemma}

\begin{proof}
Let $\lambda_1, \lambda_2$ be submodules of $\pi$ such that $\lambda_2/\lambda_1 =\tau$. Now we have a filtration on $\lambda_p$ ($p=1,2$):
\[  0 \subset (\pi_1' \cap \lambda_p) \subset (\pi_2'\cap \lambda_p) \subset \ldots \subset (\pi_r'\cap \lambda_p) =\lambda_p  .
\]
Now $\lambda_2/\lambda_1$ admits a filtration 
\[  0 \subset \gamma_1 \subset \gamma_2 \subset \ldots \subset \gamma_r ,
\]
where $\gamma_k = (\pi_{k}'\cap \lambda_2) /(\pi_k'\cap \lambda_1)$ and so
\[ \gamma_{k+1}/\gamma_k \cong \frac{(\pi_{k+1}' \cap \lambda_2)/(\pi_{k+1}'\cap \lambda_1)}{(\pi_{k}'\cap \lambda_2) /(\pi_k'\cap \lambda_1)} .\]
This implies $\pi_{k+1}' \cap \lambda_2$ and so $\pi_{k+1}'$ has irreducible subquotient isomorphic to $\tau$. 

\end{proof}

\begin{lemma} \label{lem reduction}
Let $\pi$ be a non-zero smooth representation of $G$. Let $\tau$ be a non-zero $G$-submodule of $\pi$. Suppose $\pi$ admits a filtration on 
\[  0=\pi_0 \subset \pi_1 \subset \ldots \subset \pi_r=\pi .
\] 
Then there exists a non-zero $G$-submodule $\tau'$ of $\tau$ such that
\[  \tau' \hookrightarrow \pi_{k+1}/\pi_k. 
\]
for some $k$.
\end{lemma}

\begin{proof}
Let $k$ be the smallest positive integer such that
\[  \tau \cap \pi_k \neq 0 .
\]
Define $\tau'=\tau \cap \pi_k$. The non-zero $G$-submodule $\tau' $ embeds to $\pi_k/\pi_{k-1}$ as desired.
\end{proof}

\begin{lemma} \label{lem two filtration intersect}
Let $0 \neq \pi \in \mathrm{Alg}(G)$. Suppose $\pi$ admits two $G_n$-filtrations:
\[  0=\pi_0 \subsetneq \pi_1\subset \pi_2\subset \ldots \subset \pi_r =\pi
\]
and
\[  0= \pi_0' \subsetneq \pi_1' \subset \pi_2' \subset \ldots \subset \pi_s'=\pi .
\]
Then there exists a non-zero $\tau \in \mathrm{Alg}(G)$ such that for some $i,j$,
\[  \tau \hookrightarrow \pi_{i+1}/\pi_i, \quad \tau \hookrightarrow \pi_{j+1}'/\pi_j'  .
\]

\end{lemma}

\begin{proof}
This follows by considering the smallest integer $k$ such that $\pi_1 \cap \pi'_k \neq 0$. 
\end{proof}

\end{document}